\documentclass[reqno]{amsart}
\usepackage{amsmath,amssymb,amsthm,amsrefs}
\usepackage{xcolor}

\usepackage[utf8]{inputenc}

\newtheorem{teo}{Teorema}[section]
\newtheorem{prop}[teo]{Proposition}
\newtheorem{lema}[teo]{Lemma}
\newtheorem{cor}[teo]{Corolary}
\theoremstyle{definition}
\newtheorem{defn}[teo]{Definition}
\newtheorem{obs}[teo]{Remark}

\newtheorem{ex}[teo]{Example}

\newcommand{\vu}{\vspace{.1cm}}

\begin{document}
	\allowdisplaybreaks \hyphenation{Ma-te-ma-ti-ca}

\title{Partial Actions of Weak Hopf Algebras on Coalgebras}

\author[Campos]{Eneilson Campos}
\address[Campos]{Universidade Federal do Rio Grande, Brazil}
\email{eneilsonfontes@furg.br}

\author[Fonseca]{Graziela Fonseca}
\address[Fonseca]{Universidade Federal do Rio Grande do Sul, Brazil}
\email{grazielalangone@gmail.com}

\author[Martini]{Grasiela Martini}
\address[Martini]{Universidade Federal do Rio Grande, Brazil}
\email{grasiela.martini@furg.br}

\begin{abstract}
In this work the notions of partial action of a weak Hopf algebra on a coalgebra and partial action of a groupoid on a coalgebra will be introduced, just as some important properties. An equivalence between these notions will be presented. Finally, a dual relation between the structures of partial action on a coalgebra and partial action on an algebra will be established, as well as a globalization theorem for partial module coalgebras will be presented.
\end{abstract}



\thanks{The second author was partially supported by CNPq, Brazil}

\maketitle

\section{Introduction}

The notion of partial actions of a group on a set appeared for the first time in \cite{Exelp} and was introduced by R. Exel. The purpose in such a  work was to describe $C^*$-algebras having as group the automorphisms of the unitary circle $S_1$. Afterwards, M. Dokuchaev and R. Exel defined, in \cite{Dokuchaev}, the partial action of a group on an algebra, bringing the partial actions into a purely algebraic context. From there, many results were obtained, one of them is the development of a Galois theory for partial actions of groups on rings by M. Dokuchaev, M. Ferrero and A. Paques, \cite{michamiguelpaques}. This work motivated the definition of partial actions of Hopf algebras presented by S. Caenepeel and K. Janssen in \cite{CaenJanssen}.

In \cite{Dokuchaev}, it was investigated under what conditions a partial group action has a globalization (also known as an enveloping action). Generalizing this notion, M. Alves and E. Batista  have constructed in \cite{Ab} a globalization for a partial action of a Hopf algebra on an algebra. These results became known as Globalization Theorems.

The actions theory of weak bialgebra started with G. Böhm, F. Nill and K. Szlachányi  in \cite{Bohminicio} and has been developed in works such as \cite{Caenepeel} and \cite{Wang}. Meantime, D. Bagio and A. Paques introduced in \cite{Bagio} the definition of partial groupoid action on an algebra extending the notion of partial group action. In such a work some properties presented in \cite{Dokuchaev} were also generalized.

Later in \cite{Felipeweak} the concept of partial action of a weak Hopf algebra on an algebra was introduced and the results of Hopf algebras were verified, analogously, for this new context. For example, an equivalence was established between a partial action of a groupoid $\mathcal{G}$ on a unitary algebra $A$ and the partial action of the groupoid algebra $\Bbbk \mathcal{G}$ on $A$.

A natural question that also arose was to know what happens when the partial action is no longer of a Hopf algebra on an algebra but on a coalgebra. Then, in  \cite{Glauberglobalizations} it was  developed a new theory of partial Hopf action for the context of coalgebras. Inspired by the constructions given in \cite{Batista} they obtained, among other results, new globalization theorems and constructed a partial smash coproduct generated by partial coactions of Hopf algebras on coalgebras. In this way, such a study stimulate us to go ahead and construct a theory for partial actions of weak Hopf algebras on coalgebras.

\vu

The main purpose of this work is to introduce the notion of partial action of a weak Hopf algebra on a coalgebra and to investigate how behave in this more general context the results from the classical Hopf theory.

\vu

In the first section of this paper the theoretical basis necessary for the understanding of this text is presented. All properties and results presented in this part already exist in the literature, hence the respective demonstrations will be omitted. However, the references are duly presented for further study of the subject.

\vu

In the first part of the second section the definition of a (global) action of a weak bialgebra on a coalgebra is recalled and the definitions of global action and partial action of a weak Hopf algebra on coalgebras are presented, as well as its respective properties. This section also shows a family of examples of partial actions of a weak Hopf algebra on a coalgebra. In addition, it is investigated what conditions are necessary and sufficient for a partial module coalgebra to be obtained by projection from a (global) module coalgebra.

\vu
In the third section the notion of a partial groupoid action on a coalgebra is presented as a  generalization of the notion of a partial group action on a coalgebra of E. Batista and J. Vercruysse defined in \cite{Batista}. Besides that, some important properties for the development of the theory of partial groupoid actions on coalgebras are highlighted. At the end of this section, an equivalence between the partial action of a groupoid $\mathcal{G}$ on a coalgebra $C$ and the partial action of the groupoid algebra  $\Bbbk \mathcal{G}$ on $C$ is constructed.

Section 4 of this paper is dedicated to the dualizations, that is, the dual correspondences between the respective structures of partial module algebra and partial module coalgebra are investigated. As we will see later, these results will be essential to  get the globalization theorem that will be presented in the last section.

\vu

Finally, the last section is composed of two parts. In the first part a correspondence is established between the globalization of a partial module coalgebra and the globalization of a partial module algebra via dualizations. In the second part, the necessary conditions for a partial module coalgebra to be globalizable will be investigated, extending the idea presented in \cite{Glauberglobalizations}.

\vu

Throughout, vector spaces and (co)algebras will be all considered over a fixed field $\Bbbk$. The symbol $\otimes$ will always mean $\otimes_{\Bbbk}$. $H$ will always denote a weak Hopf algebra, unless some additional conditions are assumed, and $C$ a coalgebra.

\section{Preliminaries}
\subsection{Weak Hopf Algebras}\label{wpaction} \quad In this section we will recall the notion of weak Hopf algebra from the more general concept of weak bialgebra. We will also present several important properties of these structures which will be fundamental for the development of the theory of partial actions of weak Hopf algebras on coalgebras. For more details we refer \cite{Bohminicio}, \cite{Bohmexemplo} and \cite{Bohm}.

\vu

A bialgebra $(H, m, u, \Delta, \varepsilon)$ is a
struture in which the counit $ \varepsilon $ and the coproduct $ \Delta $ are multiplicative and preserve the unit. Now, we will work with a notion similar to this one, although somewhat more generic, in the sense that it will not require all of these conditions for $ \Delta $ and $ \varepsilon $. Thus, a vector space $ H $ is said to be a \textit {weak bialgebra} if there are linear maps
$$m:H\otimes H \rightarrow H, \ \ \ \ u:\Bbbk \rightarrow H$$
and
$$\Delta:H \rightarrow H \otimes H, \ \ \ \ \varepsilon :H \rightarrow \Bbbk$$
such that $ (H, m, u) $ is an algebra, $(H,\Delta,\varepsilon)$ is a coalgebra and, in addition, the following conditions are satisfied:
\begin{enumerate}
	\item[(i)] $\Delta(hk) = \Delta(h)\Delta(k)$;
	\item[(ii)] $\varepsilon(hk\ell) = \varepsilon(hk_1)\varepsilon(k_2\ell) = \varepsilon(hk_2)\varepsilon(k_1\ell)$;
	\item[(iii)] $(1_H \otimes \Delta(1_H))(\Delta(1_H) \otimes 1_H) = (\Delta(1_H) \otimes 1_H)(1_H \otimes \Delta(1_H)) = \Delta^2(1_H)$.
\end{enumerate}
Furthermore, it will be used the notation of Sweedler \cite {Sweedler}, for $\Delta $, that is,
$$\Delta (h) = \displaystyle \sum_ {h} h_1 \otimes h_2, $$
or simply, $\Delta (h) = h_1 \otimes h_2, $ where the summation is implied. With the notation presented, the item (iii) can be rewritten as
$$1_{1'} \otimes 1_1 1_{2'} \otimes 1_2 = 1_1 \otimes 1_2 1_{1'} \otimes 1_{2'} = 1_1 \otimes 1_2 \otimes 1_3.$$
Therefore, since $\Delta$ is multiplicative and $\Delta(h) = \Delta (h1_H) = \Delta(1_Hh)$, then
\begin{eqnarray}
h_1 \otimes h_2 = h_11_1 \otimes h_21_2
= 1_1h_1 \otimes 1_2h_2. \label{4.2}
\end{eqnarray}

Given the above, $ \varepsilon $ can be used to define the following maps which are also linear:
\begin{eqnarray*}
	\varepsilon_t: H &\rightarrow& H \ \ \ \ \ \ \  \ \ \ \ \mbox{ and} \ \  \ \ 	\varepsilon_s: H \rightarrow H \\
	h &\mapsto& \varepsilon(1_1h)1_2 \ \ \ \ \ \ \ \ \ \ \ \  \ \ \ \  \ \ \	h \mapsto 1_1\varepsilon(h1_2) .
\end{eqnarray*}

These maps are said, respectively,  \textit{target map} and \textit{source map}, which give rise to the notation used. From the above, the vector spaces $H_t= \varepsilon_t(H)$ and $H_s= \varepsilon_s(H)$ can de defined.

\begin{prop}
If $H$ is a weak bialgebra and $h, k \in H$, then the following properties hold:
\begin{eqnarray}
\varepsilon_t(\varepsilon_t(h)) &=& \varepsilon_t(h)  \label{4.3}\\
\varepsilon_s(\varepsilon_s(h)) &=& \varepsilon_s(h)  \label{4.4} \\
\varepsilon(h\varepsilon_t(k)) &=& \varepsilon(hk)  \label{4.5} \\
\varepsilon(\varepsilon_s(h)k) &=& \varepsilon(hk)  \label{4.6}\\
\Delta(1_H) & \in & H_s \otimes H_t \label{4.7}\\
\varepsilon_t(h\varepsilon_t(k)) &=& \varepsilon_t(hk)  \label{4.8}\\
\varepsilon_s(\varepsilon_s(h)k) &=& \varepsilon_s(hk). \label{4.9}
\end{eqnarray}
\end{prop}

By (\ref{4.7}), $ H_t $ and $ H_s $ are isomorphic as vector spaces whenever the dimension of H is finite. On the other hand, by (\ref{4.8}) and (\ref{4.9}), it is possible to see that $\Delta(H_t) \subseteq H \otimes H_t$ and $\Delta(H_s) \subseteq H_s \otimes H$. And, more specifically,
\begin{eqnarray}
\Delta(h) &=& 1_1h \otimes 1_2, \mbox{ for all $h \in H_t$}  \label{4.10}\\
\Delta(h) &=& 1_1 \otimes h1_2,  \mbox{ for all $h \in H_s.$}\label{4.11}
\end{eqnarray}

\begin{prop}
If $H$ is a weak bialgebra, then the following properties hold, for all $h, k \in H$:
\begin{eqnarray}
h_1 \otimes \varepsilon_t(h_2) &=& 1_1h \otimes 1_2  \label{4.12}\\
\varepsilon_s(h_1) \otimes h_2 &=& 1_1 \otimes h1_2 \label{4.13} \\
h\varepsilon_t(k) &=& \varepsilon(h_1k)h_2  \label{4.14} \\
\varepsilon_s(h)k &=& k_1\varepsilon(hk_2)  \label{4.15}\\
1_{1'} \otimes \varepsilon_t(1_{2'}) \otimes 1_{3'}&=& 1_11_{1'} \otimes 1_2 \otimes 1_{2'} \label{4.17}\\
1_1 \otimes \varepsilon_s(1_2) \otimes 1_3&=& 1_1 \otimes 1_{1'} \otimes 1_21_{2'} \label{4.18}\\
\varepsilon_t(\varepsilon_t(h)k) &=& \varepsilon_t(h)\varepsilon_t(k)  \label{4.19} \\
\varepsilon_s(h\varepsilon_s(k)) &=& \varepsilon_s(h)\varepsilon_s(k).  \label{4.20}
\end{eqnarray}
\end{prop}

Such properties imply that $ H_t $ and $ H_s $ are subalgebras of $ H $ containing $ 1_H $ and that
\begin{eqnarray}
hk &=& kh,  \mbox{ for all $h \in H_t, \ k \in H_s$}. \label{4.16}
\end{eqnarray}

With such results, we are able to introduce the concept of weak Hopf algebra. Just as in the case of Hopf algebras, a new map is defined, called antipode, which is related with the counit, as one can see below.
\begin{defn}[Weak Hopf Algebra]
Let $ H $ be a weak bialgebra. We say that $ H $ is a \textit{weak Hopf algebra} if there is a linear map $ S $, called antipode, which satisfies:	\begin{enumerate}
		\item [(i)] $h_1S(h_2)=\varepsilon_t(h)$;
		\item [(ii)] $S(h_1)h_2=\varepsilon_s(h)$;
		\item [(iii)] $S(h_1)h_2S(h_3)=S(h),$
	\end{enumerate}
for all $h \in H$.
\end{defn}

As in the classical theory, the antipode of a weak Hopf algebra is \textit{anti-multiplicative}, that is, $ S (hk) = S (k) S (h) $, and \textit{anti-comultiplicative}, which means that $ S (h) _1 \otimes S (h) _2 = S (h_2) \otimes S (h_1) $. 

\begin{prop}
If $H$ is a weak Hopf algebras, the following identities hold for all $ h \in H $:
\begin{eqnarray}
\varepsilon_t(h) &=& \varepsilon(S(h)1_1)1_2 \label{4.30}\\
\varepsilon_s(h) &=& 1_1\varepsilon(1_2S(h)) \label{4.31}\\
\varepsilon_t(h) &=& S(1_1)\varepsilon(1_2h) \label{4.32}\\
\varepsilon_s(h) &=& \varepsilon(h1_1)S(1_2) \label{4.33}\\
\varepsilon_t \circ S &=& \varepsilon_t \circ \varepsilon_s = S \circ \varepsilon_s \label{4.34}\\
\varepsilon_s \circ S &=& \varepsilon_s \circ \varepsilon_t = S \circ \varepsilon_t. \label{4.35}\\
h_1 \otimes h_2S(h_3) &=& 1_1h \otimes 1_2 \  \label{4.36}\\
S(h_1)h_2 \otimes h_3 &=& 1_1 \otimes h1_2 \ \label{4.37}\\
h_1 \otimes S(h_2)h_3 &=& h1_1 \otimes S(1_2) \  \label{4.38}\\
h_1S(h_2) \otimes h_3 &=& S(1_1) \otimes 1_2h \  \label{4.39}\\
h_2S^{-1}(h_1) \otimes h_3 &=& S(\varepsilon_t(h_1)) \otimes h_2 = 1_1 \otimes 1_2h, \ \label{4.41}
\end{eqnarray}
for all $h \in H$. And,
\begin{eqnarray}
1_1S^{-1}(h) \otimes 1_2 &=& 1_1 \otimes 1_2h, \mbox{ for all $h \in H_t$} \label{4.42}\\
1_1 \otimes S^{-1}(h)1_2 &=& h1_1 \otimes 1_2, \ \mbox{for all $h \in H_s$}.\label{4.43}
\end{eqnarray}

\end{prop}

With such properties one can be shown that $S(1_H) = 1_H$, $\varepsilon \circ S = \varepsilon$, $S(H_t)=H_s$, $S(H_s)=H_t$ and $1_1 \otimes 1_2 = S(1_2) \otimes S(1_1)$. Thus, if $ H $ is a weak Hopf algebra, then $ S (H) $ also is, with the same counit and antipode.

It is known that every Hopf algebra is a weak Hopf algebra. A natural question that arises is: ``under what conditions does a weak Hopf algebra become a Hopf algebra?" For the purpose of to answer this question, the following result is presented.
\begin{prop}
	A weak Hopf algebra is a Hopf algebra if one of the following equivalent conditions is satisfied:
	\begin{itemize}
		\item [(i)] $\Delta(1_H)=1_h \otimes 1_H$;
		\item[(ii)] $\varepsilon(hk) = \varepsilon(h)\varepsilon(k);$
		\item [(iii)]$h_1S(h_2)=\varepsilon(h)1_H; $
		\item [(iv)]$S(h_1)h_2 = \varepsilon(h)1_H;$
		\item [(v)]$H_t=H_s=\Bbbk 1_H,$
	\end{itemize}
for all $h, k \in H$.	
\end{prop}

In order to construct an example of weak Hopf algebra, we present the following definition.

\begin{defn}[Groupoid] \label{grupoide}
Consider $ \mathcal{G} $ a non-empty set with a binary operation partially defined. Given $ g, h \in \mathcal{G} $, we write $ \exists gh $ whenever the product $ gh $ is set (similarly it will be used $ \nexists gh $ whenever the product is not defined). Thus, $ \mathcal{G} $ is called \textit{groupoid} if
	
	\begin{itemize}
		\item [(i)] For all $g, h, l \in \mathcal{G}$, $\exists(gh)l$ if and only if $\exists g(hl)$, and, in this case, $(gh)l = g(hl)$;
		\item [(ii)] For all $g, h, l \in \mathcal{G}$, $\exists(gh)l$ if and only if $\exists gh$ and $\exists hl$;
		\item[(iii)] For each $g \in \mathcal{G}$ there are unique elements $d(g), r(g) \in \mathcal{G}$ such that $\exists gd(g)$, $\exists r(g)g$ and $gd(g)=g=r(g)g$;
		\item[(iv)] For each $g \in \mathcal{G}$ there exists an element such that $d(g)= g^{-1}g$ and $r(g)=gg^{-1}$.
	\end{itemize}
\end{defn}

An element $ e $ is said identity in $ \mathcal{G} $ if for some $ g \in \mathcal{G} $, $ e = d (g) = r(g^{-1})$. $ \mathcal{G}_0 $ will denote the set of all identity elements of $ \mathcal{G} $. It is immediate to see that for every element $e \in \mathcal{G}_0$ the set $$\mathcal{G}_e = \{g \in \mathcal{G} \ | \ d(g)=e=r(g) \}$$ is a group. Besides that, one defines the set
$$\mathcal{G}^{2} = \{(g,h) \in \mathcal{G} \times \mathcal{G} \ | \ \exists gh\}$$
of all pairs of elements composable in $\mathcal{G}$.

As an immediate consequence of Definition \ref{grupoide} is the following result.

\begin{prop}
	Let $\mathcal{G}$ be a groupoid. Then the following properties hold:
	\begin{itemize}
	\item [(i)] $e^2 = e$, $d(e) = e = r (e) $ and $ e = e^{-1}$, for all $e \in \mathcal{G}_0$.
	\item  [(ii)] For each $g \in \mathcal{G}$, the element $ g^{-1}$  is the only one that satisfies that $d(g)= g^{-1}g$ and $r(g)=gg^{-1}$. In addition, $({g^{-1}})^{-1} = g$.
	\item [(iii)] For all $g, h \in \mathcal{G}$, $\exists gh$ if and only if $d(g)=r(h)$ and, in this case, $d(gh) = d(h)$ and $r(gh)=r(g)$;
	\item [(iv)] For all $g, h \in \mathcal{G}$, $\exists gh$ if and only if $\exists h^{-1} g^{-1}$ and, in this case, ${(gh)}^{-1} = h^{-1} g^{-1}$.
	\end{itemize}
\end{prop}

\begin{ex}[Groupoid Algebra] \label{algebradegrupoide} Let $\mathcal{G}$ be a groupoid such that the cardinality $|\mathcal{G}_{0}|$ of $\mathcal{G}_{0}$ is finite and $ \Bbbk \mathcal{G}$ the vector space with basis indexed by the elements of $\mathcal{G}$ denoted by $\lbrace \delta_g \rbrace_{g\in \mathcal{G}}$. Then, $\Bbbk \mathcal{G}$ is a weak Hopf algebra with the following structure:
	\begin{eqnarray*}
		m(\delta_g\otimes \delta_h)=\left\{
		\begin{array}{rl}
			\delta_{gh}, & \text{if $\exists gh$ },\\
			0, & \text{ otherwise }
		\end{array} \right.
	\end{eqnarray*}	
	$$u(1_{\Bbbk})=1_\mathcal{G} = \sum_{e \in \mathcal{G}_{0}} \delta_e \ \ \ \ \ \ \Delta(\delta_g)=\delta_g\otimes \delta_g$$
	$$\varepsilon(\delta_g)=1_{\Bbbk} \ \ \ \ \ \ \ \ \ \ \ \ \ \  S(\delta_g)=\delta_{g^{-1}}.$$
	
	\label{ex_grupoide}
\end{ex}

As in the Hopf case, when $H$ is a finite dimensional weak Hopf algebra the dual structure $H^{*}=Hom({H, \Bbbk})$ is a weak Hopf algebra with the convolution product given by $m(f\otimes g)(h)=(f*g)(h)=f(h_1)g(h_2)$, for all $h \in H,$ unit given by $u(1_{\Bbbk})=1_{H^*}= \varepsilon_H$, coprodut defined by the relation $\Delta(f) = f_1 \otimes f_2 \Leftrightarrow f(hk)=f_1(h)f_2(k)$, for all $h,k \in H$, counit given by $\varepsilon_{H^*}(f)=f(1_H)$, and antipode $S_{H^{*}}(f)= f \circ S_H$. Besides that, 
$(\varepsilon_t)_{H^*}(f)= f \circ (\varepsilon_t)_H$
and
$(\varepsilon_s)_{H^*}(f)= f \circ (\varepsilon_s)_H.$

The example below is presented to illustrate this result.

\begin{ex} [Dual Groupoid Algebra] Let $\mathcal{G}$ be a finite groupoid and $ (\Bbbk \mathcal{G})^{*}$  the  vector space with basis indexed by the elements of $\mathcal{G}$ given by $\lbrace p_g \rbrace_{g\in \mathcal{G}}$, where
	\begin{eqnarray*}
		p_g(\delta_h)=\left\{
		\begin{array}{rl}
			1_{\Bbbk}, & \text{if $g=h$ },\\
			0, & \text{ otherwise. }
		\end{array} \right.
	\end{eqnarray*}	
	
	Then, $(\Bbbk \mathcal{G})^{*}$ is a weak Hopf algebra with the following structure:
	\begin{eqnarray*}
		m(p_g\otimes p_h)=p_g * p_h \ \ \ \ \ \ \ 	1_{(\Bbbk \mathcal{G})^{*}} = \sum_{g \in \mathcal{G}} p_g
	\end{eqnarray*}	

	$$\Delta(p_g)=\sum_{h \in \mathcal{G}, \exists h^{-1}g} p_h \otimes p_{h^{-1}g} \ \ \ \ \varepsilon(p_g)=p_g(1_{\Bbbk \mathcal{G}}) \ \ \ \ S(p_g)=p_{g} \circ S.$$
	
	\label{ex_grupoidedual}
\end{ex}

Another important example of a weak Hopf algebra is the one presented in \cite{Bohmexemplo} where a weak Hopf algebra is generated by a finite abelian group. 

\begin{ex}\label{exemplodogrupo}
Consider $G$ a finite abelian group with order $|G|=N$, where $N$ is not a multiple of the characteristic of $\Bbbk$. If $\Bbbk G$ is considered the algebra with basis indexed by the elements of $G$ and with coalgebra structure given by:

	$$\Delta(g)=\frac{1}{N}\sum_{h \in {G}} gh  \otimes h^{-1} \ \ \ \ \varepsilon(g)=\left\{
	\begin{array}{rl}
	N, & \text{if $g=1_G$ },\\
	0, & \text{ otherwise. }
	\end{array} \right.
	$$
Then $\Bbbk G$ is a weak Hopf algebra with antipode defined by $S(g)=g$. Besides that, $\varepsilon_s(g)=\varepsilon_t(g)=g$, for all $g \in G$, what implies that $H_s=H_t=\Bbbk G$.

\end{ex}

\subsection{Partial Actions on Algebras} \quad In this section the notion of a partial action theory of weak Hopf algebras on algebras will be presented. More details can be found in \cite{Caenepeel}, \cite{Felipeweak} and \cite{Glauberglobalizations}. We start recalling the notion of $H$-module algebra, when $H$ is a weak bialgebra.

Let $H$ be a weak bialgebra and $A$ an algebra. If there exists a linear map
	\begin{eqnarray*}
		\vartriangleright: H \otimes A & \longrightarrow & A\\
		h \otimes a & \longmapsto & h\vartriangleright a
	\end{eqnarray*}
such that, for $h,k\in H$ and $a,b\in A$ satisfies,
	\begin{enumerate}
		\item [(MA1)] $1_H \vartriangleright a=a$;
		\item [(MA2)] $h \vartriangleright ab=(h_1 \vartriangleright a)(h_2 \vartriangleright b)$;
		\item [(MA3)] $h\vartriangleright(k \vartriangleright a)=hk\vartriangleright a$;
		\item[(MA4)] $h \vartriangleright 1_A = \varepsilon_t(h) \vartriangleright 1_A.$
	\end{enumerate}
then, we say that $A$ is a \textit{left $H$-module algebra}, or that $H$ acts on the algebra $A$. Similarly, we can define a right $H$-module algebra.

In \cite{Felipeweak} the authors showed that, if $H$ is a weak Hopf algebra, then, the conditions (MA1)-(MA3) imply (MA4). Moreover, it was introduced the notion of partial module algebras for the case of weak Hopf algebras.

\begin{defn} An algebra $A$ is a \textit{left partial $H$-module algebra} if there exists a linear map
	\begin{eqnarray*}
		\cdot: H \otimes A & \longrightarrow & A\\
		h \otimes a & \longmapsto & h\cdot a
	\end{eqnarray*}
such that, for $h,k\in H$ and $a,b\in A$,
	\begin{enumerate}
		\item [(PMA1)] $1_H \cdot a=a$;
		\item [(PMA2)] $h\cdot ab=(h_1\cdot a)(h_2\cdot b)$;
		\item [(PMA3)] $h\cdot(k\cdot a)=(h_1\cdot 1_A)(h_2k\cdot a)$.
	\end{enumerate}
\end{defn}

In this case we say that there is a partial action of $H$ on $A$. If the additional condition $h\cdot (k\cdot a)=(h_1k\cdot a)(h_2\cdot 1_A)$ is required, we say that the partial action is \textit{symmetric}. Analogously, one defines a right partial module algebra.

\vu 

Similar to the case of partial Hopf actions, in \cite{Felipeweak} the authors defined globalization for a partial module algebra as below.

\begin{defn} \label{globallizacaoalgebra}
Let $A$ be a left partial $H$-module algebra. We say that  $(B, \theta)$ is a globalization of $A$ if $B$ is a left $H$-module algebra via $\triangleright: H \otimes B \rightarrow B$, and
	
	\begin{itemize}
		\item [(i)] $\theta: A \rightarrow B$ is an algebra monomorphism such that $\theta(A)$ is a right ideal of $B$;
		\item [(ii)] The partial action on $A$ is equivalent to the induced partial action by $\triangleright$ on $\theta(A)$, what means, $\theta(h \cdot a)= h \cdot \theta(a)= \theta(1_A)(h \triangleright \theta(a))$;
		\item[(iii)] $B$ is a left $H$-module algebra generated by $\theta(A)$, i.e, $B = H \triangleright \theta (A)$.
	\end{itemize}
\end{defn}

Furthermore, every left partial module algebra admits a globalization, as proved in \cite{Felipeweak}.

\section{Partial Actions on Coalgebras}
\quad \
We begin this section by introducing the notion presented by G. Böhm in \cite{Bohminicio}, of actions of weak bialgebras on coalgebras. 

\begin{defn}\label{bohmincio}
	We say that $C$ is a \textit{left $H$-module coalgebra} if there is a linear map
	$$\begin{array}{rl}
	\vartriangleright: H \otimes C &\rightarrow C \\
	h \otimes c &\mapsto h \vartriangleright c
	\end{array}$$
	such that, for any $h, k \in H$ e $c \in C$:
	\begin{enumerate}
		\item[(\label{MC2}MC1)] $1_H \vartriangleright c = c;$
			\item[(\label{MC1}MC2)] $\Delta (h \vartriangleright c) = h_1 \vartriangleright c_1 \otimes h_2 \vartriangleright c_2;$
			\item[(\label{MC3}MC3)] $h \vartriangleright k \vartriangleright c = hk \vartriangleright c;$
\item[(MC4)] $\varepsilon(h \vartriangleright c) = \varepsilon( \varepsilon_s(h) \vartriangleright c).$
	\end{enumerate}
	\end{defn}

In this case, we also say that $ H $ acts on the coalgebra $ C $. One defines similarly a right $ H $-module coalgebra.

If $H$ is a weak Hopf algebra the condition (MC4) is a consequence of (MC1)-(MC3), as follows.
\begin{prop}\label{propglobal}
If $H$ is a weak Hopf algebra and a linear map
	$$\begin{array}{rl}
	\vartriangleright: H \otimes C &\rightarrow C \\
	h \otimes c &\mapsto h \vartriangleright c
	\end{array}$$
satisfies (MC1)-(MC3) then, (MC4) is automatically satisfied.
\end{prop}
\begin{proof}Given $c \in C$ and $h \in H$,
	$$\begin{array}{rl}
	\varepsilon(h \vartriangleright c)&=\varepsilon(h \vartriangleright 1_H \vartriangleright c)\\
	&=\varepsilon(1_{H_1} \vartriangleright c_1) \varepsilon(h1_{H_2} \vartriangleright c_2)\\
	&= \varepsilon(\varepsilon_s(h_1) \vartriangleright c_1) \varepsilon(h_2 \vartriangleright c_2) \\
	&=\varepsilon(S(h_1)h_2 \vartriangleright c_1) \varepsilon(h_3 \vartriangleright c_2) \\
	&=\varepsilon(S(h_1) \vartriangleright h_2 \vartriangleright c_1) \varepsilon(h_3 \vartriangleright c_2)\\
	&=\varepsilon(S(h_1)h_2 \vartriangleright c)\\
	&=\varepsilon(\varepsilon_s(h) \vartriangleright c).
	\end{array}$$
\end{proof}

\begin{ex} \label{Exemploglobal} 
	
(i) Every weak Hopf algebra $H$ is a left $H$-module coalgebra via its multiplication.

(ii) Consider $\Bbbk \mathcal{G}$  the groupoid algebra, then $\Bbbk \mathcal{G}$ is a left $\Bbbk \mathcal{G}$-module coalgebra via
$$\begin{array}{rl}
\triangleright : \Bbbk \mathcal{G} \otimes \Bbbk \mathcal{G} &\rightarrow \Bbbk \mathcal{G}\\
\delta_h \otimes \delta_g &\mapsto \delta_gS(\delta_h).
\end{array}$$
\end{ex}

\subsection{Partial Module Coalgebra}
\quad
Throughout this section we will present the concept of partial action of a weak Hopf algebra on a coalgebra. We will also show some properties and some examples that provide the basis for this theory.

\begin{defn}\label{modulocoalgebraparcial}
	We say that $ C $ is a \textit{left partial $H$-module coalgebra} if there is a linear map
	$$\begin{array}{rl}
	\cdot: \ H \otimes C &\rightarrow C \\
	h \otimes c &\mapsto h \cdot c
	\end{array}$$
	such that, for any $h, k \in H$ and $c \in C$:
	\begin{enumerate}
	\item[(\label{MCP2}PMC1)] $1_H \cdot c = c$;
		\item[(\label{MCP1}PMC2)] $\Delta (h \cdot c) = h_1 \cdot c_1 \otimes h_2 \cdot c_2$;
			\item[(\label{MCP3}PMC3)] $h \cdot k \cdot c = (hk_1 \cdot c_1) \varepsilon(k_2 \cdot c_2).$
	\end{enumerate}
In addition, we say that $C$ is a \textit{left symmetric partial module coalgebra} if it still satisfies $h \cdot k \cdot c = \varepsilon(k_1 \cdot c_1) (hk_2 \cdot c_2).$
\end{defn}

In this case, we also say that $ H $ acts partially on the coalgebra $ C $. Analogously, one can define a right symmetric partial $H$-module coalgebra. Throughout this section whenever we refer to a partial $H$-module coalgebra it will be considered on the left. In addition, the notion of partial module coalgebra  generalizes the one of module coalgebra.

\begin{ex}\label{exemplossemlambida}
	Consider $\Bbbk \mathcal{G}$ the groupoid algebra and $e \in \mathcal{G}_0$. Then, the group algebra $\Bbbk \mathcal{G}_e $ is a partial $\Bbbk \mathcal{G}$-module coalgebra via
	$$\delta_g \cdot h=\left \{ \begin{array}{rl}  h, \ if \ g = e \\
	0, if \ g \neq e.
	\end{array}
	\right. $$
\end{ex}

The next result characterizes the necessary and sufficient condition for a partial $H$-module coalgebra to be a global $H$-module coalgebra.
\begin{prop}
Let $C$ a partial $H$-module coalgebra. Then, $C$ is a $H$-module coalgebra if and only if	$ \varepsilon(h \cdot c) = \varepsilon( \varepsilon_s(h) \cdot c),$ for all $h \in H$ and $c \in C$.
\end{prop}
 \begin{proof}
Let $C$ be a partial $H$-module coalgebra such that $ \varepsilon(h \cdot c) = \varepsilon( \varepsilon_s(h) \cdot c),$ then, it is to enough to show that $h \cdot k \cdot c = hk \cdot c$, for all $h,k \in H$ and $c \in C$. This follows because,
\begin{eqnarray*}
h \cdot k \cdot c &=& (hk_1 \cdot c_1) \varepsilon(k_2 \cdot c_2)\\
&=& (hk_1 \cdot c_1) \varepsilon( \varepsilon_s(k_2) \cdot c_2)\\
&=& (hk_1 \cdot c_1) \varepsilon(S(k_2)k_3 \cdot c_2)\\
&\stackrel{(\ref{4.38})}{=} &(hk 1_{H_1} \cdot c_1) \varepsilon(S(1_{H_2}) \cdot c_2)\\
&\stackrel{(\ref{4.38})}{=}& (h_1k_1 \cdot c_1) \varepsilon(S(h_2k_2)h_3k_3 \cdot c_2)\\
&=& (h_1k_1 \cdot c_1) \varepsilon(\varepsilon_s(h_2k_2) \cdot c_2)\\
&=& (h_1k_1 \cdot c_1) \varepsilon(h_2k_2 \cdot c_2)\\
&=& hk \cdot c.
\end{eqnarray*}
The converse is immediate by Proposition \ref{propglobal}.
\end{proof}
The next results have the objective to present some properties that involve $ H_t $ and $ H_s $. The idea is to explore how the actions of these algebras can be cha\-rac\-te\-ri\-zed.
\begin{prop}\label{Ht}
	Suppose that $C$ is a partial $H$-module coalgebra, then, for all $h \in H_t$, $k \in H$ and $c \in C$ we have:
	\begin{enumerate}
		\item[(i)] $h \cdot k \cdot c=hk \cdot c;$
		\item[(ii)] $\varDelta (h \cdot c) = h \cdot c_1 \otimes c_2;$
		\item[(iii)] $ \varepsilon(h \cdot c) = \varepsilon( \varepsilon_s(h) \cdot c)$
	\end{enumerate}	
\end{prop}
\begin{proof}
		Given $h \in H_t$, $k \in H$ and $c \in C$,\\
(i)
	$$\begin{array}{rl}
		h \cdot k \cdot c &= (hk_1 \cdot c_1) \varepsilon(k_2 \cdot c_2)\\
		&\stackrel{(\ref{4.16})}{=} (1_{H_1}hk_1 \cdot c_1) \varepsilon(1_{H_2}k_2 \cdot c_2)\\
		&\stackrel{(\ref{4.10})}{=} (h_1k_1 \cdot c_1) \varepsilon(h_2k_2 \cdot c_2)\\
		&=hk \cdot c
		\end{array}$$
		
(ii)	
		$$\begin{array}{rl}
		\varDelta (h \cdot c)&\stackrel{(\ref{4.11})}{=} (h1_{H_1} \cdot c_1) \otimes (1_{H_2} \cdot c_2) \\
		&\stackrel{(i)}{=} (h \cdot 1_{H_1} \cdot c_1) \otimes (1_{H_2} \cdot c_2) \\
		&\stackrel{(\ref{4.11})}{=} h \cdot c_1 \otimes c_2,
		\end{array}$$

(iii)
	$$\begin{array}{rl}
	\varepsilon(h \cdot c) &= \varepsilon(1_{H_1} \cdot c_1)\varepsilon(h \cdot 1_{H_2} \cdot c_2) \\
	&\stackrel{(ii)}{=} \varepsilon(1_{H_1} \cdot c_1)\varepsilon(h1_{H_2} \cdot c_2) \\
	&\hspace{-0.15cm} \stackrel{(\ref{4.13})}{=} \varepsilon(\varepsilon_s(h_1) \cdot c_1)\varepsilon(h_2 \cdot c_2) \\
	&= \varepsilon(S(h_1)h_2 \cdot c_1)\varepsilon(h_3 \cdot c_2) \\
	&\hspace{-0.15cm} \stackrel{(\ref{4.10})}{=} \varepsilon(S(1_{H_1}h)1_{H_2} \cdot c_1)\varepsilon(1_{H_3} \cdot c_2) \\
	&= \varepsilon(S(h)\varepsilon_s(1_{H_1}) \cdot c_1)\varepsilon(1_{H_2} \cdot c_2) \\
	&\hspace{-0.10cm} \stackrel{(\ref{4.7})}{=} \varepsilon(S(h)1_{H_1} \cdot c_1)\varepsilon(1_{H_2} \cdot c_2) \\
	&\hspace{-0.15cm} \stackrel{(\ref{modulocoalgebraparcial})}{=} \varepsilon(S(h) \cdot 1_H \cdot c) \\
	&= \varepsilon(S(h) \cdot c) \\
	&= \varepsilon(S(h)\varepsilon_s(1_H) \cdot c) \\
	&= \varepsilon(S(h)S(1_{H_1})1_{H_2} \cdot c) \\
	&\stackrel{(\ref{4.10})}{=}\varepsilon( S(h_1)h_2 \cdot c)\\
	&= \varepsilon( \varepsilon_s(h) \cdot c).
	\end{array}$$
\end{proof}

\begin{cor}\label{ht}
	If $H = H_t$, then every partial $H$-module coalgebra is a $H$-module coalgebra.
\end{cor}

Still analyzing this situation, the next proposition illustrates the case of a sym\-met\-ryc partial $H$-module coalgebra.
\begin{prop}
Suppose that $C$ is a symmetric partial $H$-module coalgebra, then, for all $h \in H_s$, $k \in H$ and $c \in C$ we have:
	\begin{enumerate}
		\item[(i)] $h \cdot k \cdot c=hk \cdot c$;
		\item[(ii)] $\varDelta (h \cdot c) = c_1 \otimes h \cdot c_2.$
	\end{enumerate}	
\end{prop}

\begin{proof} The proof is analogous to that of Proposition \ref{Ht}.
%
%
\end{proof}

Besides that, the condition $\varepsilon (h \cdot c) = \varepsilon (\varepsilon_s(h) \cdot c) $ which characterizes a left  (global) $H$-module coalgebra is naturally satisfied when $ h \in H_s $.
\begin{cor}\label{hs}
	If $H = H_s$, then every partial $H$-module coalgebra is a $H$-module coalgebra.
\end{cor}

By Corolary \ref{ht} and \ref{hs} it follows that the weak Hopf algebra $\Bbbk G$ presented in Example \ref{exemplodogrupo} always acts globally on a coalgebra.

\subsection{Characterizing Actions via $\lambda$}
\quad \
The purpose of this section is to explore a specific family of examples of partial module coalgebra. To do, consider $\lambda \in Hom_\Bbbk(H,\Bbbk)$. 

We say that $C$ is a \textit{$H$-module coalgebra via $\lambda$} if
$$\begin{array}{rl}
\vartriangleright: \ H \otimes C &\rightarrow C \\
h \otimes c &\mapsto \lambda(h) c
\end{array}$$
defines an action of a weak Hopf algebra $H$ on the coalgebra $C$.

\begin{prop}\label{lambdaglobal}
	$C$ is a $H$-module coalgebra via $\lambda$ if and only if
	\begin{enumerate}
		\item [(i)] $\lambda(1_H) = 1_\Bbbk$;
		\item [(ii)] $\lambda(h) = \lambda(h_1)\lambda(h_2)$;
		\item [(iii)] $\lambda(h)\lambda(k) = \lambda(hk)$,
	\end{enumerate}
	for all $h,k \in H$.
\end{prop}

\begin{proof}
It follows by Definition \ref{bohmincio}.
\end{proof}

By Proposition \ref{propglobal} if $C$ is a $H$-module coalgebra via $\lambda$, then, $ \lambda $ also satisfies
$\lambda(h) = \lambda(\varepsilon_s(h))$, for all $h \in H$. Besides this, note that if we take $\varepsilon$ in place of $\lambda$, we have that $C$ is a $H$-module coalgebra via $ \varepsilon $ if and only if $H$ is an Hopf algebra in the usual sense.
\begin{ex}

(i) Let $\mathcal{G}$ be a groupoid given by the disjoint union of finite groups $G_1,...,G_n$ and fix $G_j$ one of them. Considering $\Bbbk \mathcal{G}$ the groupoid algebra, we can  define $\lambda \in Hom_\Bbbk(\Bbbk \mathcal{G}, \Bbbk)$ as follows
		$$\lambda(\delta_g)=\left \{ \begin{array}{rl}  1, \ if \ g \in G_j \\
		0, if \ g \notin G_j.
		\end{array}
		\right. $$
Then, $\lambda$ satisfies the conditions of Proposition \ref{lambdaglobal}.

(ii) Let $\Bbbk G$ the weak Hopf algebra given in Example \ref{exemplodogrupo}.  We can  define $\lambda \in$ $Hom_\Bbbk(\Bbbk {G}, \Bbbk)$ as follows
$$\lambda(g) = 1, \mbox{ for all $g \in G$.} $$
Then, $\lambda$ satisfies the conditions of Proposition \ref{lambdaglobal}.
	
 \end{ex}

\vu

\quad

In the partial case, we say that $C$ is a \textit{partial $H$-module coalgebra via $\lambda$} if
$$\begin{array}{rl}
\cdot: \ H \otimes C &\rightarrow C \\
h \otimes c &\mapsto \lambda(h) c
\end{array}$$
defines a partial action of $ H $ on the coalgebra $ C $.

\begin{prop}\label{lambdaparcial}
	$C$ is a partial $H$-module coalgebra via $\lambda$ if and only if
	\begin{enumerate}
		\item [(i)] $\lambda(1_H) = 1_\Bbbk$;
		\item [(ii)] $\lambda(h)\lambda(k) =\lambda(hk_1)\lambda(k_2),$
	\end{enumerate}
	for all $h,k \in H$.
\end{prop}

\begin{proof}
	It follows by Definition \ref{modulocoalgebraparcial}.
\end{proof}

Besides that, notice that the partial action of $H$ on $C$ is symmetric if and only if $\lambda$ satisfies additionally that $\lambda(h)\lambda(k) = \lambda(k_1)\lambda(hk_2)$, for all $h,k \in H$.

\begin{obs}
 If $\lambda $ characterizes a partial action of a weak Hopf algebra $ H $ on a coalgebra $C$, then, $\lambda $ satisfies $ \lambda(h) = \lambda(\varepsilon_s( h)), $ for all $ h \in H $, if and only if $ C $ is a $ H $-module coalgebra via $\lambda$.
\end{obs}

\begin{ex}

(i)  \label{exlambda} Consider $ \mathcal{G} $ a groupoid and $e$ an element in $\mathcal{G}_0$.Then  $\mathcal{G}_e $ defines a partial action of $ \Bbbk \mathcal{G} $ on a coalgebra $ C $ via $ \lambda $. It is enough to see that
	$$ \lambda(\delta_g) =\left \{ \begin{array}{rl}  1, \ if \ g \in \mathcal{G}_e \\
	0, if \ g \notin \mathcal{G}_e.
	\end{array}
	\right. $$
	satisfies the conditions of Proposition \ref{lambdaparcial}.
	 
(ii) \cite{Felipeweak} \label{exlambda2} Consider $ \mathcal{G} $ the groupoid given by the disjoint union of the groups $ G_1, ..., G_n $, and fix $ G_j $ one of these groups. Then any subgroup $ V $ of $ G_j $ defines a partial action of $ \Bbbk \mathcal{G} $ on a coalgebra $ C $ via $ \lambda $, it is enough to see that
$$ \lambda(\delta_g) =\left \{ \begin{array}{rl}  1, \ if \ g \in V \\
0, if \ g \notin V.
\end{array}
\right. $$
satisfies the conditions of Proposition \ref{lambdaparcial}.	

\end{ex}

As in \cite{Felipeweak}, we can also characterize the weak Hopf algebra action of $ \Bbbk \mathcal{G} $ on $\Bbbk $, where $ \mathcal{G}_0 $ is finite and $\Bbbk$ is seen as a coalgebra.
\begin{prop}The ground field $\Bbbk$ is a partial $\Bbbk \mathcal{G}$-module coalgebra via $\lambda$ if and only if $V=\{g \in \mathcal{G} ; \ \delta_g \cdot 1_{\Bbbk} = 1_{\Bbbk}= \delta_{d(g)}\cdot 1_{\Bbbk}\}$ is a group.
\end{prop}

\begin{proof}
	Suppose that $\Bbbk$ is a partial $\Bbbk \mathcal{G}$-module coalgebra via $\delta_g \cdot 1_{\Bbbk} = \lambda(\delta_g ).$ Then,
	\begin{eqnarray*}
		(\delta_g \cdot 1_{\Bbbk})&=& (1_{\Bbbk \mathcal{G}}\cdot 1_{\Bbbk})(\delta_g \cdot 1_{\Bbbk})\\
		&=&\sum_{e \in \mathcal{G}_0} (\delta_e \cdot 1_{\Bbbk})(\delta_g \cdot 1_{\Bbbk})\\
		&\stackrel{\ref{lambdaparcial}}{=}&\sum_{e \in \mathcal{G}_{0}}(\delta_e \delta_g \cdot 1_{\Bbbk})(\delta_g \cdot 1_{\Bbbk})\\
		&=&(\delta_g \cdot 1_{\Bbbk})^2.
	\end{eqnarray*}
	
Since $\Bbbk$ is a field, $(\delta_g \cdot 1_{\Bbbk})=1$ or $(\delta_g \cdot 1_{\Bbbk})=0$. Now we proceed to prove that $V$ is a group. Indeed:

(i) $V\neq\varnothing$, because $\lambda(1_{\Bbbk \mathcal{G}})=1_{\Bbbk}$.

(ii) If $g,h \in V$, then $gh \in V$.

Let $g,h \in V$, then $1_{\Bbbk}=\lambda(\delta_g)\lambda(\delta_h) =\lambda(\delta_g \delta_h)\lambda(\delta_h) =\lambda(\delta_{g}\delta_h).$
Therefore, there is $gh$, otherwise $\lambda(\delta_{g}\delta_h)=\lambda(0)=0$. Thus, $\lambda(\delta_{gh})=\lambda(\delta_g\delta_h)=1_{\Bbbk}$, that is, $gh \in V$.

(iii)If $g \in V$, we have necessarily that $g^{-1} \in V$. 

For this, notice that $1_{\Bbbk}=\lambda(\delta_{d(g)})\lambda(\delta_g)	=\lambda(\delta_{g^{-1}}\delta_g)\lambda(\delta_g) =\lambda(\delta_{g^{-1}})\lambda(\delta_g) \stackrel{g \in V}{=}\lambda(\delta_{g^{-1}}).$

(iv) For all $h,g \in V$, $d(g)=r(g)=d(h)$ . 

Given $g,h \in V$, we already know that there exists $gh^{-1}$, then, we have that
\begin{eqnarray}\label{oi}
d(g)=r(h^{-1})=h^{-1}h=d(h).
\end{eqnarray}
Since (\ref{oi}) goes for all $h \in V$, in particular, for $h=g^{-1}$,  $d(g)=r(g)$. 

What shows that $V$ is a group. The converse is immediate.
\end{proof}

Similarly, there is an analogous result when it is considered the action of the weak Hopf algebra $ (\Bbbk \mathcal{G})^*$ on the field $ \Bbbk $, for $ \mathcal{G} $ a finite groupoid and $ \Bbbk $ is seen as a coalgebra.
\begin{prop}
The ground field $\Bbbk$ is a partial $(\Bbbk \mathcal{G})^*$-module coalgebra via $\lambda$ if and only if $V=\{g \in \mathcal{G} ; \ p_g \cdot 1_{\Bbbk} \neq 0 \mbox{ e } p_{g^{-1}}\cdot 1_{\Bbbk}\ \neq 0\}$ is a group and the characteristic of $\Bbbk$ does not divide the cardinality of $V$. In this case, the action is defined by
	$$\lambda(p_g)=\left\{ \begin{array}{cc}
	\dfrac{1}{|V|}, & \mbox{ if $g \in V$}\\
	0, & \mbox{ otherwise}
	\end{array}
	\right..$$
\end{prop}

\begin{proof}

	We will show that $\Bbbk$ is a partial $(\Bbbk \mathcal{G})^{*}$-module coalgebra via $\lambda$. To do this notice that

(i) $\lambda(1_{(\Bbbk \mathcal{G})^{*}})= 1_{\Bbbk}$. Indeed,
		\begin{eqnarray*}
			\lambda(1_{(\Bbbk \mathcal{G})^{*}})
			=\sum_{g \in \mathcal{G}} \lambda( p_g)
			= |V| \dfrac{1}{|V|}
			= 1_{\Bbbk}.
		\end{eqnarray*}	
		
	(ii) $\lambda(p_g)\lambda({p_h}) = \lambda(p_g (p_h)_1)\lambda((p_h)_2)$. Indeed, 
		\begin{eqnarray*}
			\lambda(p_g)\lambda({p_h})=\left\{
			\begin{array}{rl}
				\dfrac{1}{|V|^2}	, & \text{if $g,h \in V$ }\\
				0, & \text{ otherwise }
			\end{array} \right..
		\end{eqnarray*}	
		
		On the other hand
		\begin{eqnarray*}
			\lambda(p_g (p_h)_1)\lambda((p_h)_2)&=& \sum_{l \in \mathcal{G}, \exists l^{-1}h} \lambda (p_g p_l) \lambda(p_{l^{-1}h})\\
			&=& \sum_{l \in \mathcal{G}, \exists l^{-1}h} \delta_{g,l}\lambda (p_g) \lambda (p_{l^{-1}h})\\
			&=&\lambda (p_g) \lambda (p_{g^{-1}h})\\
			&=&\left\{
			\begin{array}{rl}
				\dfrac{1}{|V|^2}	, & \text{if $g,h \in V$ }\\
				0, & \text{ otherwise }
			\end{array} \right.,
		\end{eqnarray*}	
	
	where $\delta_{g,l} = 1_\Bbbk$ if $g=l$ and $\delta_{g,l} =0$ if $g\neq l$.
	
	The converse is immediate.
	\end{proof}

\subsection{Induced Partial Action for Module Coalgebra} \label{induzida}
\quad \
A pertinent question is whether a partial action of a weak Hopf algebra  $H$ on a coalgebra can be constructed  from a $H$-module coalgebra. That is, given a global action of $H$ on a coalgebra $C$, is it possible to construct a partial action, induced by the global one, of $H$ on a subcoalgebra of $C$? The purpose of this section is to characterize the necessary and sufficient conditions for this to occur, and to investigate when this constructed action is a global one.

Suppose that $H$ acts on a coalgebra $C$ via
$$\begin{array}{rl}
\vartriangleright: H \otimes C &\rightarrow C \\
h \otimes c &\mapsto h \vartriangleright c
\end{array}$$
and consider $D \subseteq C$ a subcoalgebra of $C$ such that there is a linear projection $\pi: C \rightarrow C$ that satisfies $\pi(C) = D.$ Under these conditions we have the following result.
\begin{prop}\label{acaoinduzida}
	$D$ is a partial $H$-module coalgebra via
	$$\begin{array}{rl}
	\cdot: \ H \otimes D &\rightarrow D \\
	h \otimes d &\mapsto \pi(h \vartriangleright d)
	\end{array}$$
	if and only if the projection $\pi$ satisfies
	\begin{enumerate}
		\item[(i)] $(\pi \otimes \pi) \varDelta(h \vartriangleright d) = \varDelta(\pi(h \vartriangleright d))$;
		\item[(ii)]$
		\pi(h \vartriangleright \pi(k \vartriangleright d)) =\varepsilon(\pi(k_2 \vartriangleright d_2)) \pi (h k_1 \vartriangleright d_1),$
		\end{enumerate}
for all $h,k \in H$ and $d \in D.$	In addition, the symmetric condition is equivalent to $\pi(h \vartriangleright \pi(k \vartriangleright d))= \varepsilon(\pi(k_1 \vartriangleright d_1)) \pi (h k_2 \vartriangleright d_2)$, for all $h,k \in H$ and $d \in D.$
\end{prop}

\begin{proof}
If $D$ is a partial $H$-module coalgebra via $(h \cdot d) = \pi(h \vartriangleright d)$, by Definition \ref{modulocoalgebraparcial}, it is easy to see that $\pi$ satisfies the above conditions.

Conversely, suppose that $\pi$ satisfies these two conditions, then, for all $d \in D$ and $h,k \in H$,
	
	\begin{enumerate}
		\item [(i)] $	1_H \cdot d = \pi(1_H \vartriangleright d) = \pi(d) =d.$
		\item [(ii)] $\varDelta(h \cdot d)= \varDelta(\pi(h \vartriangleright d))
		=(\pi \otimes \pi) \varDelta(h \vartriangleright d)
		= h_1 \cdot d_1 \otimes h_2 \cdot d_2.$
		\item [(iii)]$	h \cdot k \cdot d =\pi(h \vartriangleright \pi(k \vartriangleright d))
		=\varepsilon(\pi(k_2 \vartriangleright d_2)) \pi (h k_1 \vartriangleright d_1)
	=(hk_1 \cdot d_1) \varepsilon(k_2 \cdot d_2).$
		\item [(iv)] To show the symmetry condition it is enough to see that 
$$h \cdot k \cdot d =\pi(h \vartriangleright \pi(k \vartriangleright d)) =\varepsilon(\pi(k_1 \vartriangleright d_1)) \pi (h k_2 \vartriangleright d_2) = (hk_2 \cdot d_2) \varepsilon(k_1 \cdot d_1).$$
	\end{enumerate}
\end{proof}

We call the partial action constructed in Proposition \ref{acaoinduzida} by \textit{partial induced action}. It is immediate that the induced action is global if and only if $\varepsilon(h \cdot d) = \varepsilon(\varepsilon_s(h) \cdot d)$ what is equivalent to $\varepsilon(\pi (h \vartriangleright d)) = \varepsilon(\pi (\varepsilon_s(h) \vartriangleright d)),$ for all $h \in H$ and $d \in D$.
\begin{ex}
	
(i) Consider $\Bbbk \mathcal{G}$ with $\mathcal{G}$ a groupoid. We already know that $\Bbbk \mathcal{G}$ is a $\Bbbk \mathcal{G}$-module coalgebra via its multiplication by Example \ref{Exemploglobal} (i). Assume that $e \in \mathcal{G}_0$ and $h \in \mathcal{G}_e$, thus the vector space $D=<\delta_h>_{\Bbbk} $ is a subcoalgebra of $\Bbbk \mathcal{G}$. Define
		$$ \pi(\delta_g) =\left \{ \begin{array}{rl}  \delta_g, \ if \ \delta_g \in D \\
		0, if \ \delta_g \notin D.
		\end{array}
		\right. $$
	Therefore, $D$ is a symmetric partial $\Bbbk \mathcal{G}-$module coalgebra.	

(ii) In view of Example \ref{Exemploglobal} (ii), it is known that $\Bbbk \mathcal{G}$ is a $\Bbbk \mathcal{G}$-module coalgebra via $\delta_h \triangleright \delta_g =\delta_gS(\delta_h)$. Then, the projection
		$$ \pi(\delta_g) =\left \{ \begin{array}{rl}  \delta_g, \ if \ \delta_g \in D \\
		0, if \ \delta_g \notin D.
		\end{array}
		\right. $$
		characterizes a partial action of $\Bbbk \mathcal{G}$ on $D = <\delta_l>_\Bbbk$, for a fixed $l$ of $\mathcal{G}$. In addition, note that this partial action is not global since $\varepsilon(\pi (\delta_g \vartriangleright d)) \neq \varepsilon(\pi (\varepsilon_s(\delta_g) \vartriangleright d)).$ To see this, it is enough to take $d = \delta_g= \delta_l$.
\end{ex}

\section{Groupoid Action on Coalgebra}
\quad \
In this section we will work with a groupoid $\mathcal{G}$ such that $|\mathcal{G}_0|$ is finite.

\begin{defn}\label{acaodegrupoideemcoalgebra}
A \textit{partial action of a groupoid $\mathcal{G}$ on a coalgebra} $C$ is a family of subcoalgebras $\{C_g\}_{g \in \mathcal{G}}$ and coalgebra isomorphisms $\{\theta_g: C_{g^{-1}} \rightarrow C_g\}_{g \in \mathcal{G}}$ such that
	
	\begin{enumerate}
		\item [(i)]
		For each $g \in \mathcal{G}$, there exists a linear projection	$P_g: C \rightarrow C$ such that $P_g(C)=C_g$ and, for all $c \in C$ 
		\begin{eqnarray}
		(P_g \otimes P_g) \varDelta(c) &=& \varDelta(P_g(c)); \label{comulti}\\
	P_g(c)&=& P_{r(g)} (c_1) \varepsilon(P_g(c_2))\nonumber\\
	&=& P_{r(g)} (c_2) \varepsilon(P_g(c_1))\label{quisicoalgebra};
	\end{eqnarray}
			\item [(ii)]
		For each $e \in \mathcal{G}_0$, $\theta_e = I_{C_e};$
		\item [(iii)] The following properties hold:
		\begin{eqnarray}
		P_g \circ P_h &=& P_h \circ P_g, \mbox{ for all }g, h \in \mathcal{G};\label{1}\\
		\theta_{h^{-1}} \circ P_h \circ P_{g^{-1}} &=& P_{{(gh)}^{-1}} \circ \theta_{h^{-1}} \circ P_h, \mbox{ for all } (g, h) \in \mathcal{G}^2;\label{2}\\
		\theta_g \circ \theta_h \circ P_{{(gh)}^{-1}} \circ P_{h^{-1}} &=& \theta_{gh} \circ P_{{(gh)}^{-1}} \circ P_{h^{-1}}, \mbox{ for all }(g, h) \in \mathcal{G}^2. \label{3}
		\end{eqnarray}
	\end{enumerate}
\end{defn}

\begin{obs}
		\begin{itemize}
		\item [1.] From the item (i), $P_g(c)=c$, for all $c\in C_g$. 
		\item[2.] From the item (i), we can conclude that $Im P_g \subseteq Im P_{r(g)}$, then
		\begin{eqnarray}
		P_{r(g)} \circ P_g = P_g, \mbox{ for all }g \in \mathcal{G} \label{4}.
		\end{eqnarray}
		\end{itemize}
\end{obs}

\begin{lema}\label{2.5.3}
Suppose that a groupoid $\mathcal{G}$ acts partially on a coalgebra. Then:
	
	\begin{enumerate}
		\item [(i)]
		 For all $g \in \mathcal{G}$, $\theta_{r(g)} \circ \theta_g = \theta_g$ and $\theta_{r(g)} \circ P_g = P_g;$
				
		\item [(ii)]
		For all $g \in \mathcal{G}$, $\theta_{g^{-1}}= {(\theta_g)}^{-1}$;
		\item [(iii)]
		$P_{g^{-1}} \circ \theta_h = \theta_h \circ P_{({gh)}^{-1}} \circ I_{C_{h^{-1}}},$ for all $(g, h) \in \mathcal{G}^2.$
	\end{enumerate}
\end{lema}

\begin{proof}
	(i) Let $g \in \mathcal{G}$, then $r(g)g = g$, and, $r(g) = gg^{-1}$. Substituting $g$ for $r(g)$ and $h$ for $g$ in the equation (\ref{3}), we have $\theta_{r(g)} \circ \theta_g = \theta_{g},$ since $P_g$ is a projection, for each $g\in\mathcal{G}$. Furthermore, for each $c \in C$, there exists $d \in C_{g^{-1}}$ such that $P_g(c) = \theta_g(d).$ Therefore, for all $ c \in C,$
			$$\begin{array}{l}
			\theta_{r(g)} \circ P_g(c) = \theta_{r(g)}(\theta_g(d))
			=\theta_g(d)=P_g(c).
			\end{array}$$
			
		(ii) In (\ref{3}), replacing $h$ by $g^{-1}$ we have	
		$$\theta_g \circ \theta_{g^{-1}} \circ P_{r(g)} \circ P_g = \theta_{r(g)} \circ P_{r(g)} \circ P_g.$$
Since (\ref{4}) holds, then $\theta_g \circ \theta_{g^{-1}} \circ P_g = \theta_{r(g)} \circ P_g$. In view of (i) it follows that  $\theta_g \circ \theta_{g^{-1}} =I_{C_g}$. Besides that, since $\theta_g$ is a bijection, thus $\theta_{g^{-1}}= {(\theta_g)}^{-1}$.
		
		(iii)
		We know that, by (\ref{1}) and (\ref{2})
		$$	\theta_{h^{-1}} \circ P_{g^{-1}} \circ P_h = P_{{(gh)}^{-1}} \circ \theta_{h^{-1}} \circ P_h.$$
		Then, by applying $\theta_h$ on both sides of this equation, we have
		$$ P_{g^{-1}} \circ P_h
		= \theta_h \circ P_{{(gh)}^{-1}} \circ \theta_{h^{-1}} \circ P_h.$$
		
		By hypothesis, for each $d \in C_{h^{-1}}$ there exists $c \in C$ such that $\theta_h(d) = P_h(c)$. Then, $\mbox{ for all } (g, h) \in \mathcal{G}^2,$
		$$\begin{array}{rl}
		P_{g^{-1}} \circ \theta_h  = \theta_h \circ P_{{(gh)}^{-1}} \circ \theta_{h^{-1}} \circ \theta_h =\theta_h \circ P_{({gh)}^{-1}} \circ I_{C_{h^{-1}}}.
		\end{array}$$
	\end{proof}

\begin{prop}\label{ida}
Assume that $\mathcal{G}$ is a groupoid that acts partially on the coalgebra $C = \bigoplus_ {e \in \mathcal{G}_0} C_e$, then
	$$\begin{array}{rl}
	\cdot : \Bbbk \mathcal{G} \otimes C & \rightarrow C \\
	\delta_g \otimes c &\mapsto \theta_g(P_{g^{-1}}(c))
	\end{array}$$
	
	defines a symmetric partial action of $\Bbbk \mathcal{G}$ on $C$.
\end{prop}

\begin{proof}It is straightforward to check that $\cdot$ is well defined as a linear map.
	
(i)	Observe that since $C = \bigoplus_ {e \in \mathcal{G}_0} C_e$, thus for each $c \in C$, $c = \sum_{e \in \mathcal{G}_0} P_e(c).$ Then, using Lemma \ref{2.5.3}:
$$\begin{array}{rl}
1_{\Bbbk \mathcal{G}} \cdot c  =\displaystyle \sum_{e \in \mathcal{G}_0} \delta_e \cdot c
=\displaystyle \sum_{e \in \mathcal{G}_0} \theta_e(P_{e}(c))
=\displaystyle \sum_{e \in \mathcal{G}_0} P_{e}(c)
= c, \mbox{ for all $c \in C.$}
\end{array}$$

	
		
%
(ii) For all $c \in C$ and $g \in \mathcal{G}$,
		$$\begin{array}{rl}
		\varDelta(\delta_g \cdot c) &= \varDelta(\theta_g (P_{g^{-1}} (c))) \\
		&= \theta_g({P_{g^{-1}}(c_1)}) \otimes \theta_g({P_{g^{-1}}(c_2)})\\
		&= \delta_g \cdot c_1 \otimes \delta_g \cdot c_2.
		\end{array}$$
		
(iii) To show that $\delta_g \cdot \delta_h \cdot c = (\delta_{gh}\cdot c_1) \varepsilon(\delta_{h}\cdot c_2),
		=(\delta_{gh}\cdot c_2) \varepsilon(\delta_{h}\cdot c_1)$ we need to analyze two cases.
		
Case 1)
			Suppose that $(g,h) \notin \mathcal{G}^2$, then $\nexists gh$, thus we know that $\delta_{gh}=0$. Therefore, it is enough to show that $\delta_g \cdot \delta_h \cdot c =0$. First of all, note that, for each $e,f \in \mathcal{G}_0$
			$$ P_f \circ P_e =\left \{ \begin{array}{rl}  P_e, \ if \ e=f \\
			0, if \ e \neq f,
			\end{array}
			\right. $$
			by $C = \bigoplus_ {e \in \mathcal{G}_0} C_e$ and (\ref{1}). Since $d(g) \neq r(h)$, $P_{d(g)} \circ P_{r(h)} = 0$. In addition,
			$$\begin{array}{rl}
			P_{g^{-1}} \circ P_h &\stackrel{(\ref{4})}{=} P_{r({g^{-1}})} \circ P_{g^{-1}} \circ P_{r(h)} \circ P_h\\
			&\stackrel{(\ref{1})}{=} P_{g^{-1}} \circ P_{r({g^{-1}})} \circ P_{r(h)} \circ P_h\\
			&=P_{g^{-1}} \circ P_{d(g)} \circ P_{r(h)} \circ P_h\\
			&=0.
			\end{array}$$
	Therefore, $\delta_g \cdot \delta_h \cdot c = \theta_g(P_{g^{-1}}({\theta_h(P_{h^{-1}}(c))))}=0. $
			
Case 2)	In this case $ (g, h) \in \mathcal{G}^2$. First, we will show that
			$$			\begin{array}{rl}
			\delta_{gh} \cdot P_{h^{-1}}(c) &= (\delta_{gh} \cdot c_1) \varepsilon(\delta_{h} \cdot c_2)\\
			&= (\delta_{gh} \cdot c_2) \varepsilon(\delta_{h} \cdot c_1).
			\end{array}	
			$$
			To do this, it is enough to take $eh^{-1}= h^{-1}$ and note that
			$$
			\begin{array}{rl}
			\delta_{gh} \cdot P_{h^{-1}}(c) &\stackrel{(\ref{quisicoalgebra})}{=} \delta_{gh} \cdot P_{e}(c_1) \varepsilon(P_{h^{-1}}(c_2))\\
			&{=} \delta_{gh} \cdot \theta_e(P_{e}(c_1)) \varepsilon(P_{h^{-1}}(c_2))\\
			&= (\delta_{gh} \cdot \delta_e \cdot c_1) \varepsilon(P_{h^{-1}}(c_2))\\
			&\stackrel{(*)}{=} (\delta_{gh} \cdot c_1) \varepsilon(\delta_e \cdot c_2) \varepsilon(P_{h^{-1}}(c_3))\\
			&{=}(\delta_{gh} \cdot c_1) \varepsilon(P_{e}(c_2)) \varepsilon(P_{h^{-1}}(c_3))\\
			&\stackrel{(\ref{quisicoalgebra})}{=} (\delta_{gh} \cdot c_1) \varepsilon(P_{h^{-1}}(c_2))\\
			&{=} (\delta_{gh} \cdot c_1) \varepsilon(\theta_h(P_{h^{-1}}(c_2)))\\
			&=(\delta_{gh} \cdot c_1) \varepsilon(\delta_{h} \cdot c_2),\\
			\end{array}	
			$$
			
		where, in $(*)$	we are using that
	$$\begin{array}{rl}
	\delta_{gh} \cdot \delta_e \cdot c &=\theta_{gh}(P_{{(gh)}^{-1}}(\theta_e(P_e(c))))\\
	&\stackrel{(\ref{comulti})}{=}\theta_{gh}(P_{{(gh)}^{-1}}(\theta_e(P_{e}(c_1))) \varepsilon(\theta_e(P_{e}(c_2)))\\
	&=\theta_{gh}(P_{{(gh)}^{-1}}(P_{e}(c_1))) \varepsilon(\delta_{e}\cdot c_2)\\
	&\stackrel{(\ref{1})}{=}\theta_{gh}(P_e(P_{{(gh)}^{-1}}(c_1))) \varepsilon(\delta_{e}\cdot c_2)\\
	&\stackrel{(\ref{4})}{=}\theta_{gh}(P_{{(gh)}^{-1}}(c_1)) \varepsilon(\delta_{e}\cdot c_2)\\
	&= (\delta_{gh}\cdot c_1) \varepsilon(\delta_e\cdot c_2),
	\end{array}$$
Similarly, we show $ \delta_{gh} \cdot P_{h^{-1}}(c) =(\delta_{gh} \cdot c_2) \varepsilon(\delta_{h} \cdot c_1)$ using that $\delta_{gh} \cdot \delta_e \cdot c = (\delta_{gh}\cdot c_2) \varepsilon(\delta_e\cdot c_1).$
Now, we will show
				$$\begin{array}{rl}
				\delta_g \cdot \delta_h \cdot c &= (\delta_{gh} \cdot c_1) \varepsilon(\delta_{h} \cdot c_2)\\
				&= (\delta_{gh} \cdot c_2) \varepsilon(\delta_{h} \cdot c_1).
				\end{array}$$
				
Indeed, for each $c \in C$ and $g,h \in \mathcal{G}$,
				
			$$\begin{array}{rl}
				\delta_g \cdot \delta_h \cdot c &=\theta_{g}(P_{{g}^{-1}}(\theta_h(P_{h^{-1}}(c))))\\
				&\stackrel{\ref{2.5.3}}{=}\theta_{g}(\theta_h(P_{{(gh)}^{-1}}(P_{h^{-1}}(c))))\\
				&\stackrel{(\ref{3})}{=}\delta_{gh} \cdot (P_{h^{-1}}(c))\\
				&{=} (\delta_{gh} \cdot c_1) \varepsilon(\delta_{h} \cdot c_2)\\
				&{=} (\delta_{gh} \cdot c_2) \varepsilon(\delta_{h} \cdot c_1).
				\end{array}$$

\end{proof}

\begin{teo}
The following statements are equivalent:	
	\begin{enumerate}
		\item [(i)]There exists a partial action $\theta = (\{\theta_g\}, \{C_g\})_{g \in \mathcal{G}}$ of $\mathcal{G}$ on $C$ such that $C =\displaystyle \bigoplus_ {e \in \mathcal{G}_0} C_e$.
		
		\item [(ii)] $C$ is a symmetric partial $\Bbbk \mathcal{G}$-module coalgebra.
	\end{enumerate}
	
\end{teo}

\begin{proof}
	
	(i) $\Rightarrow$ (ii) It follows by Proposition \ref{ida}.
	
	(ii) $\Rightarrow$ (i) Suppose that $C$ is a symmetric partial $\Bbbk \mathcal{G}$-module coalgebra. We will proceed by steps:
	
Step 1) For each $g \in \mathcal{G}$, there exists a linear projection	$P_g: C \rightarrow C$ such that $P_g(C)=C_g$ and (\ref{comulti}), (\ref{quisicoalgebra}) hold.\\Define $P_g(c) = \varepsilon(\delta_{g^{-1}} \cdot c_1)(\delta_{e} \cdot c_2) = \varepsilon(\delta_{g^{-1}} \cdot c_2)(\delta_{e} \cdot c_1)$, for all $c \in C$. By the symmetry of the partial action we know that, for each $g \in \mathcal{G}$ and for $e$ such that $e=gg^{-1}$,
	$$	\begin{array}{rl}
	\delta_g \cdot \delta_{g^{-1}} \cdot c = \varepsilon(\delta_{g^{-1}} \cdot c_1)(\delta_{g}\delta_{g^{-1}} \cdot c_2)=\varepsilon(\delta_{g^{-1}} \cdot c_1)(\delta_{e} \cdot c_2).
	\end{array}$$
	Similarly, $\delta_g \cdot \delta_{g^{-1}} \cdot c =\varepsilon(\delta_{g^{-1}} \cdot c_2)(\delta_{e} \cdot c_1).$ 
	
	Now take $C_g = P_g(C)$, for each $g \in \mathcal{G}$. Moreover, $P_g(P_g(c)) = P_g(c)$, for all $c \in C$:
			$$\begin{array}{rl}
			P_g(P_g(c)) &= (\delta_{e} \cdot \delta_{e} \cdot c_1) \varepsilon(\delta_{g^{-1}}\cdot \delta_{e} \cdot c_2) \varepsilon(\delta_{g^{-1}} \cdot c_3)\\
			&= (\delta_{e}\delta_e \cdot c_1) \varepsilon(\delta_{e} \cdot c_2) \varepsilon(\delta_{g^{-1}}\delta_e \cdot c_4) \varepsilon(\delta_{e} \cdot c_3) \varepsilon(\delta_{g^{-1}} \cdot c_5)\\
&=(\delta_{e} \cdot c_1) \varepsilon(\delta_{e} \cdot c_2) \varepsilon(\delta_{g^{-1}} \cdot c_4) \varepsilon(\delta_{e} \cdot c_3) \varepsilon(\delta_{g^{-1}} \cdot c_5)\\
			&= (\delta_{e} \cdot c_1) \varepsilon(\delta_{g^{-1}} \cdot c_2)\\
			&=P_g(c).
			\end{array}$$
			
Observe that	
$$\begin{array}{rl}
			P_e(c_1) \varepsilon(P_g(c_2))&= (\delta_e \cdot c_1) \varepsilon(P_g(c_2))\\
			&= (\delta_e \cdot c_1) \varepsilon(\delta_e \cdot c_2) \varepsilon(\delta_{g^{-1}} \cdot c_3)\\
			&= (\delta_e \cdot c_1) \varepsilon(\delta_{g^{-1}} \cdot c_2)\\
			&=P_g(c),
			\end{array}$$
			similarly, $ P_e(c_2) \varepsilon(P_g(c_1))=P_g(c).$ And, for each $c \in C$ and $g \in \mathcal{G}$,
			$$\begin{array}{rl}
			\varDelta(P_g(c)) &=(\delta_e \cdot c_1 \otimes \delta_e \cdot c_2)\varepsilon(\delta_{g^{-1}} \cdot c_3) \varepsilon(\delta_{g^{-1}} \cdot c_4)\\
			&{=}(\delta_e \cdot c_1 \otimes \delta_e \cdot c_3)\varepsilon(\delta_{g^{-1}} \cdot c_2) \varepsilon(\delta_{g^{-1}} \cdot c_4)\\
			&=P_g(c_1) \otimes P_g(c_2)\\
			&=(P_g \otimes P_g)\varDelta(c).\\
			\end{array}$$
	
Step 2)	For each $e \in \mathcal{G}_0$, $\theta_e = I_{C_e},$ and for all $g \in \mathcal{G}$, $\theta_{g}$ is a coalgebra isomorphism.\\Define $\theta_g(P_{g^{-1}}(c)) = \delta_g \cdot P_{g^{-1}}(c)$, for all $c \in C$ and $g \in \mathcal{G}$. It is immediate that  $\theta_e = I_{C_e}$, for all $e \in \mathcal{G}_0$. Besides that, taking $ l $ such that $lg^{-1}= g^{-1}$,
			$$\begin{array}{rl}
			\theta_g(P_{g^{-1}}(c)) &= (\delta_g \cdot c_1) \varepsilon(\delta_l \cdot c_2)\\
			&= (\delta_g \cdot c_2) \varepsilon(\delta_l \cdot c_1).
			\end{array}$$
		Then,
			$$\begin{array}{rl}
			(\theta_g \otimes \theta_g) \varDelta(P_{g^{-1}}(c))&= (\delta_g \cdot c_2) \varepsilon(\delta_l \cdot c_1) \otimes (\delta_g \cdot c_4) \varepsilon(\delta_l \cdot c_3)\\
			&= (\delta_g \cdot c_3) \varepsilon(\delta_l \cdot c_2) \otimes (\delta_g \cdot c_4) \varepsilon(\delta_l \cdot c_1) \\
		&= \varDelta(\delta_g \cdot c_2)\varepsilon(\delta_l \cdot c_1)\\
		&=\varDelta(\theta_g(P_{g^{-1}}))(c).
					\end{array}$$
In addition, $\varepsilon(\theta_g(P_{g^{-1}}(c)))=\varepsilon(P_{g^{-1}}(c))$, for all $c \in C$ and $g \in \mathcal{G}$. To show that $\theta_g$ is injective, suppose that $\theta_g(P_{g^{-1}}(c))= 0$ and $l$ is such that $lg^{-1}= g^{-1}$. Thus,
			$$\begin{array}{rl}
			P_{g^{-1}}(c)&=(\delta_l \cdot c_1)\varepsilon(\delta_g \cdot c_2) \\
			&=\varepsilon(\delta_l \cdot c_1)(\delta_l \cdot c_2)\varepsilon(\delta_g \cdot c_3) \\
			&=\varepsilon(\delta_l \cdot c_1)(\delta_{g^{-1}} \delta_g \cdot c_2)\varepsilon(\delta_g \cdot c_3)\\
			&=\varepsilon(\delta_l \cdot c_1)(\delta_{g^{-1}} \cdot \delta_g \cdot c_2)\\
			&=\delta_{g^{-1}} \cdot (\delta_g \cdot c_2)\varepsilon(\delta_l \cdot c_1)\\
			&=0.
			\end{array}$$
			
For the surjective, consider $c \in C$ and $g \in \mathcal{G}$, then
			$$\begin{array}{rl}
			P_g(c)&=(\delta_e \cdot c_1)\varepsilon(\delta_{g^{-1}} \cdot c_2)\\
				&=\varepsilon(\delta_e \cdot c_1) (\delta_e \cdot c_2) \varepsilon(\delta_{g^{-1}} \cdot c_3)\\
				&=\varepsilon(\delta_e \cdot c_1) (\delta_g \cdot \delta_{g^{-1}} \cdot c_2)\\
				&=\theta_g(\varepsilon(\delta_e \cdot c_1) (\delta_{g^{-1}} \cdot c_2))\\
				&=\theta_{g}((\delta_{g^{-1}} \cdot c_3) \varepsilon(\delta_g \delta_{g^{-1}} \cdot c_1) \varepsilon(\delta_{g^{-1}} \cdot c_2))\\							
			&= \theta_g((\delta_{g^{-1}} \cdot c_3) \varepsilon(\delta_{g^{-1}} \cdot c_2) \varepsilon(\delta_g \cdot \delta_{g^{-1}} \cdot c_1))\\								
			&=\theta_{g}( (\delta_{l{g^{-1}}} \cdot c_3) \varepsilon(\delta_{g^{-1}} \cdot c_2) \varepsilon(\delta_g \cdot \delta_{g^{-1}} \cdot c_1))\\
			&=\theta_{g} ((\delta_l \cdot \delta_{g^{-1}} \cdot c_2) \varepsilon(\delta_g \cdot \delta_{g^{-1}} \cdot c_1))\\								
			&=\theta_{g}(	P_{g^{-1}}(\delta_{g^{-1}} \cdot c)),								
		\end{array}$$
			where $e=gg^{-1}$ and $l$ is such that $lg^{-1}=g^{-1}$. Therefore, $\theta_{g}$ is a coalgebra isomorphism, for every $g \in \mathcal{G}$.

Step 3) The equation (\ref{1}) holds.\\
Consider $f$ an element such that $fh=h$ and $e$  such that $eg=g.$ Then
		$$\begin{array}{rl}
		P_h(P_g(c))&= P_h(P_e(c_2)) \varepsilon(P_g(c_1))\\
		&= P_h(\delta_e \cdot c_3) \varepsilon(\delta_{g^{-1}} \cdot c_1) \varepsilon(\delta_e \cdot c_2)\\
		&= P_h(\delta_e \cdot c_2) \varepsilon(\delta_{g^{-1}} \cdot c_1) \\
		&= (\delta_f \cdot \delta_e \cdot c_2) \varepsilon(\delta_{h^{-1}} \cdot \delta_e \cdot c_3) \varepsilon(\delta_{g^{-1}} \cdot c_1) \\
		&= \left\{\begin{array}{rl}  &\hspace{-0.5cm}0,  \ se \ e \neq f \\
		&\hspace{-0.5cm}(\delta_e \cdot c_2) \varepsilon(\delta_{h^{-1}} \cdot c_3) \varepsilon(\delta_{g^{-1}} \cdot c_1), \ if \ e=f
		\end{array} \right.\\
		&= (\delta_e \cdot \delta_f \cdot c_2) \varepsilon(\delta_{g^{-1}} \cdot \delta_f \cdot c_1) \varepsilon(\delta_{h^{-1}} \cdot c_3) \\
		&= P_g(P_f(c_1)) \varepsilon(P_h(c_2))\\
	&=	P_g(P_h(c)).
		\end{array} $$
		
Step 4) The equation (\ref{2}) holds.\\Consider $f $ such that $ fh = h $ and $ l $ such that $ lg ^ {- 1} = g^{-1}$. Therefore,
	$$\begin{array}{rl}
		\theta_{h^{-1}} ( P_{g^{-1}} ( P_h(c))) &= \theta_{h^{-1}} ( P_{g^{-1}} ( \delta_f \cdot c_1)) \varepsilon(\delta_{h^{-1}} \cdot c_2)\\
		&= (\delta_{h^{-1}} \cdot \delta_l \cdot  \delta_f \cdot c_1) \varepsilon(\delta_g \cdot \delta_f \cdot c_2) \varepsilon(\delta_{h^{-1}} \cdot c_3)\\
		&\stackrel{(*)}{=} (\delta_{h^{-1}} \cdot \delta_l \cdot  \delta_l \cdot c_1) \varepsilon(\delta_g \cdot \delta_l \cdot c_2) \varepsilon(\delta_{h^{-1}} \cdot c_3)\\
		&= (\delta_{h^{-1}} \cdot \delta_l \cdot c_1) \varepsilon(\delta_g \cdot \delta_l \cdot c_2) \varepsilon(\delta_{h^{-1}} \cdot c_3)\\
		&= (\delta_{h^{-1}}\delta_l \cdot c_1) \varepsilon(\delta_g \delta_l \cdot c_2) \varepsilon(\delta_l \cdot c_3) \varepsilon(\delta_{h^{-1}} \cdot c_4)\\
		&= (\delta_{{h^{-1}}} \cdot c_1) \varepsilon(\delta_{g} \cdot c_2) \varepsilon(\delta_l \cdot c_3) \varepsilon(\delta_{h^{-1}} \cdot c_4),
		\end{array}$$
		in $(*)$ we are using that $l = r(g^{-1}) = d(g)= r(h) = f$.
		
		On the other side, note that if $k$ is such that $k{h}^{-1}={h}^{-1}$, then also satisfies $k{(gh)}^{-1}={(gh)}^{-1}$. Thus, using that $l = f = r(h)$,
		\begin{eqnarray*}
			& \ & P_{{(gh)}^{-1}} (\theta_{h^{-1}}( P_h (c))) \\
			&=& P_{{(gh)}^{-1}} (\delta_{h^{-1}} \cdot  \delta_l \cdot c_1) \varepsilon(\delta_{h^{-1}} \cdot c_2)\\
			&=& P_{{(gh)}^{-1}} (\delta_{{h^{-1}}} \cdot c_1) \varepsilon(\delta_l \cdot c_2) \varepsilon(\delta_{h^{-1}} \cdot c_3)\\
			&=&(\delta_k \cdot \delta_{{h^{-1}}} \cdot c_1) \varepsilon(\delta_{gh} \cdot \delta_{{h^{-1}}} \cdot c_2) \varepsilon(\delta_l \cdot c_3) \varepsilon(\delta_{h^{-1}} \cdot c_4)\\
			&=&(\delta_k \delta_{{h^{-1}}} \cdot c_1) \varepsilon(\delta_{h^{-1}} \cdot c_2)\varepsilon(\delta_{gh} \cdot \delta_{{h^{-1}}} \cdot c_3) \varepsilon(\delta_l \cdot c_4) \varepsilon(\delta_{h^{-1}} \cdot c_5)\\
&=&(\delta_{{h^{-1}}} \cdot c_1) \varepsilon(\delta_{h^{-1}} \cdot c_2) \varepsilon(\delta_{ghh^{-1}} \cdot c_3) \varepsilon(\delta_l \cdot c_4) \varepsilon(\delta_{h^{-1}} \cdot c_5)\\
			&=&(\delta_{{h^{-1}}} \cdot c_1)  \varepsilon(\delta_{g} \cdot c_2) \varepsilon(\delta_l \cdot c_3) \varepsilon(\delta_{h^{-1}} \cdot c_4).
		\end{eqnarray*}

Step 5) The equation (\ref{3}) holds.\\Consider $e$ such that $kh^{-1}= h^{-1}$, thus $k{(gh)}^{-1}={(gh)}^{-1}$, then,
		\begin{eqnarray*}
			\theta_g ( \theta_h ( P_{{(gh)}^{-1}} ( P_{h^{-1}}(c)))) &=& \theta_g (\theta_h ( \delta_k \cdot  \delta_k \cdot  c_1)) \varepsilon(\delta_{gh} \cdot  c_2)\varepsilon(\delta_h \cdot  c_3) \\
			&=& \theta_g ( \delta_h  \cdot  \delta_k \cdot  c_1) \varepsilon(\delta_{gh} \cdot  c_2)\varepsilon(\delta_h \cdot  c_3) \\
			&=&(\delta_g \cdot \delta_{h} \cdot  c_2) \varepsilon(\delta_k \cdot  c_1) \varepsilon(\delta_{gh} \cdot  c_3)\varepsilon(\delta_h \cdot  c_4) \\
			&=&(\delta_{gh} \cdot  c_3) \varepsilon(\delta_h \cdot  c_2) \varepsilon(\delta_k \cdot  c_1) \varepsilon(\delta_{gh} \cdot  c_4)\varepsilon(\delta_h \cdot  c_5) \\
			&= &(\delta_{gh} \cdot  c_3) \varepsilon(\delta_h \cdot  c_2) \varepsilon(\delta_k \cdot  c_1)\varepsilon(\delta_h \cdot  c_4) \\
			&=& (\delta_{gh} \cdot  c_2) \varepsilon(\delta_h \cdot  c_3) \varepsilon(\delta_k \cdot  c_1)\varepsilon(\delta_h \cdot  c_4) \\
			&=& (\delta_{gh} \cdot  c_2) \varepsilon(\delta_h \cdot  c_3) \varepsilon(\delta_k \cdot  c_1)\\
		&=& (\delta_{gh} \cdot  c_2)\varepsilon(\delta_k \cdot  c_1) \varepsilon(\delta_{gh} \cdot  c_3)\varepsilon(\delta_h \cdot  c_4) \\	
	&=&( \delta_{gh}  \cdot  \delta_k \cdot  c_1) \varepsilon(\delta_{gh} \cdot  c_2)\varepsilon(\delta_h \cdot  c_3) \\
	&=& \theta_{gh} ( \delta_k \cdot  \delta_k \cdot  c_1) \varepsilon(\delta_{gh} \cdot  c_2)\varepsilon(\delta_h \cdot  c_3) \\	
	&=&\theta_{gh} ( P_{{(gh)}^{-1}} ( P_{h^{-1}}(c)).
\end{eqnarray*}

Step 6)$C =\bigoplus_ {e \in \mathcal{G}_0} C_e$.\\Indeed,
		$$\begin{array}{rl}
		c = 1_{\Bbbk \mathcal{G}} \cdot c
		=\displaystyle \sum_{e \in \mathcal{G}_0} \delta_e \cdot c
		=\displaystyle \sum_{e \in \mathcal{G}_0} \theta_e (P_e(c))
		=\displaystyle \sum_{e \in \mathcal{G}_0} P_e(c).
		\end{array}	$$
Thus, $C= \sum_{e \in \mathcal{G}_0} ImP_e=  \sum_{e \in \mathcal{G}_0} C_e.$ Besides that, since
		$$\delta_e \cdot \delta_f \cdot c =
		\left \{ \begin{array}{rl}  0, \ if \ e\neq f \\
		\delta_e \cdot c, \ if \ e = f,
		\end{array}
		\right. $$
consider $d \in ImP_e \cap \sum_{f \in \mathcal{G}_0 \ f\neq e} ImP_f,$ then, there exist $b, c \in C$ such that $d= P_e(c) = \sum_{f \in \mathcal{G}_0 \ f\neq e} P_f(b)$. Thus,
		$$\begin{array}{rl}
		d &= P_e(P_e(c))\\
		&= P_e(\sum_{f \in \mathcal{G}_0 \ f\neq e} P_f(b))\\
		&= \sum_{f \in \mathcal{G}_0 \ f\neq e} \delta_e \cdot \delta_f \cdot b\\
		&=0.
		\end{array}$$
	Therefore, $ImP_e \cap  \sum_{f \in \mathcal{G}_0 \ f\neq e} ImP_f= \{0\}.$

\end{proof}

\section{Dualization}
\quad
Our purpose in this section is to generalize the dualization results presented in \cite{Glauberglobalizations} for the case of weak Hopf algebras.

\begin{teo}\label{teodual}
Let $ H $ be a weak Hopf algebra and $ C $ a left partial $H$-module coalgebra via
\begin{eqnarray*}
		\cdot : H \otimes C &\rightarrow& C\\
		h \otimes c &\mapsto&h \cdot c.
	\end{eqnarray*}
	Then, $C^*$ is a right partial $H$-module algebra via
	\begin{eqnarray*}
		\leftharpoonup : C^* \otimes H &\rightarrow& C^*\\
		\alpha \otimes h &\mapsto&(\alpha \leftharpoonup h),
	\end{eqnarray*}
	where $(\alpha \leftharpoonup h)(c)=\alpha (h \cdot c), \mbox{ for all } c \in C.$
	
Moreover, if $C$ is a symmetric left partial $H$-module coalgebra, then $C^*$ is a symmetric right partial $H$-module algebra.
\end{teo}

\begin{proof}
We will show that $C^*$ is a partial $H$-module algebra. Given $\alpha, \beta \in C^*$ and $h, k \in H$ we obtain:

(i) $	(\alpha \leftharpoonup 1_H)(c)= \alpha(1_H \cdot c)=\alpha(c),$ for all $c \in C$.	Then, $\alpha \leftharpoonup 1_H= \alpha$.	

(ii) $(\alpha \beta \leftharpoonup h) = (\alpha \leftharpoonup h_1)( \beta  \leftharpoonup h_2)$, since,
		\begin{eqnarray*}
			(\alpha \beta \leftharpoonup h)(c) &=& (\alpha \beta) (h \cdot c)\\
			&=& \alpha ((h \cdot c)_1)\beta ((h \cdot c)_2)\\
			&\stackrel{(PMC2)}{=}& \alpha (h_1 \cdot c_1)\beta (h_2 \cdot c_2)\\
			&=&(\alpha \leftharpoonup h_1)(c_1)(\beta \leftharpoonup h_2)(c_2)\\
			&{=}&((\alpha \leftharpoonup h_1)(\beta \leftharpoonup h_2))(c).
		\end{eqnarray*}
for all $c \in C$.	

(iii) $(\alpha \leftharpoonup h) \leftharpoonup g= (\alpha \leftharpoonup hg_1)(1_{C^*}   \leftharpoonup g_2)$, indeed, for all $c \in C$,
		\begin{eqnarray*}
			((\alpha \leftharpoonup h) \leftharpoonup k)(c) &=& (\alpha \leftharpoonup h) (k \cdot c)\\
			&=& \alpha (h \cdot  (k \cdot c))\\
			&\stackrel{(PMC3)}{=}& \alpha (hk_1 \cdot c_1)\varepsilon (k_2 \cdot c_2)\\
			&=&(\alpha \leftharpoonup hk_1)(c_1)(\varepsilon \leftharpoonup k_2)(c_2)\\
		&=&((\alpha \leftharpoonup hk_1)(1_{C^*} \leftharpoonup k_2))(c).\\
		\end{eqnarray*}

Suppose now that $C$ is a symmetric left partial $H$-module coalgebra. Then, for all $c \in C$,
\begin{eqnarray*}
		((\alpha \leftharpoonup h) \leftharpoonup k)(c) &=& (\alpha \leftharpoonup h) (k \cdot c)\\
		&=& \alpha (h \cdot (k \cdot c))\\
		&=& \varepsilon (k_1 \cdot c_1) \alpha (hk_2 \cdot c_2)\\
		&=&(\varepsilon \leftharpoonup k_1)(c_1)(\alpha \leftharpoonup hk_2)(c_2)\\
	&=&((1_{C^*} \leftharpoonup k_1)(\alpha \leftharpoonup hk_2))(c).\\
	\end{eqnarray*}
Therefore,  $C^*$ is a symmetric right partial $H$-module algebra.

\end{proof}

\quad

As follows, we have the converse of Theorem \ref{teodual}.

\begin{teo}\label{teodual2}
	Let $H$ be a weak Hopf algebra and $C^*$ a right partial $H$-module algebra via
	\begin{eqnarray*}
		\leftharpoonup : C^* \otimes H &\rightarrow& C^*\\
		\alpha \otimes h &\mapsto&(\alpha \leftharpoonup h).
	\end{eqnarray*}
	Then, $C$ is a left partial $H$-module coalgebra via
	\begin{eqnarray*}
		\cdot : H \otimes C &\rightarrow& C\\
		h \otimes c &\mapsto&h \cdot c.
	\end{eqnarray*}
	such that $(\alpha \leftharpoonup h)(c)=\alpha (h \cdot c),$ for all  $\alpha \in C^*$.
	
Moreover, if $C^*$ is a symmetric right partial $H$-module algebra, then $C$ is a symmetric left partial $H$-module coalgebra.
\end{teo}

\begin{proof}
We will show that $C$ is a partial $H$-module coalgebra. Given $c \in C$ and $h, k \in H$, we obtain:	
	
(i) $1_H \cdot c=c$, since for all $\alpha \in C^*$, $\alpha(1_H \cdot c)=(\alpha \leftharpoonup 1_H)(c)=\alpha(c)$.

(ii) $\Delta(h \cdot c) = (h_1 \cdot c_1) \otimes ( h_2 \cdot c_2)$, indeed, for all $\alpha, \beta \in C^*$,
		\begin{eqnarray*}
			(\alpha \otimes \beta) (\Delta(h \cdot c))
			&=&(\alpha \beta)(h \cdot c) \\
			&=&(\alpha \beta \leftharpoonup h)(c)\\
			&=&[(\alpha \leftharpoonup h_1)(\beta \leftharpoonup h_2)](c)\\
			&=&(\alpha \leftharpoonup h_1)(c_1)(\beta \leftharpoonup h_2)(c_2)\\
			&=&\alpha (h_1 \cdot c_1) \beta (h_2 \cdot c_2)\\
			&=&(\alpha \otimes \beta )[ (h_1 \cdot c_1) \otimes (h_2 \cdot c_2)].\\
		\end{eqnarray*}	
	
(iii) $h \cdot (k \cdot c)= (hk_1 \cdot c_1)\varepsilon(k_2 \cdot c_2)$, since for all $\alpha \in C^*$,
		\begin{eqnarray*}
			\alpha (h \cdot (k \cdot c))&=&(\alpha \leftharpoonup h)(k \cdot c)\\
			&=& ((\alpha \leftharpoonup h) \leftharpoonup k)(c)\\
			&=&[(\alpha \leftharpoonup hk_1)(1_{C^*} \leftharpoonup k_2)](c)\\
			&=&(\alpha \leftharpoonup hk_1)(c_1)(\varepsilon \leftharpoonup k_2)(c_2)\\
			&=&\alpha(hk_1 \cdot c_1)\varepsilon (k_2 \cdot c_2).\\
		\end{eqnarray*}

Suppose now that $C^*$ is a symmetric right partial $H$-module algebra. Then, for every $\alpha \in C^*$,

	\begin{eqnarray*}
		\alpha(h \cdot (k \cdot c))&=&(\alpha \leftharpoonup h)(k \cdot c)\\
		&=& (\alpha \leftharpoonup h \leftharpoonup k)(c)\\
		&=&[(1_{C^*} \leftharpoonup k_1)(\alpha \leftharpoonup hk_2)](c)\\
	&=&(\varepsilon \leftharpoonup k_1)(c_1)(\alpha \leftharpoonup hk_2)(c_2)\\
		&=&\varepsilon(k_1 \cdot c_1)\alpha (hk_2 \cdot c_2).\\
	\end{eqnarray*}

Therefore, $C$ is a symmetric left partial $H$-module coalgebra.
\end{proof}

\section{Globalization for Partial Module Coalgebra}

As we have seen, it is possible, under certain conditions, to induce a partial $H$-module coalgebra from a global one. In this section, we will see the converse of this situation. To do this, we start with the following definition.

\begin{defn}[\textit{Globalization for partial module coalgebra}] \label{globalmod}
	Let $H$ be a weak Hopf algebra and $(C, \leftharpoonup)$ a right partial $H$-module coalgebra. A globalization of $C$ is a triple $(D, \theta, \pi)$, such that
	\begin{enumerate}
		\item [(i)] $D$ is a right $H$-module coalgebra via
$\blacktriangleleft: D \otimes H \rightarrow D$;
		
		\item [(ii)]$\theta: C \rightarrow D$ is a coalgebra monomorphism;
		
		\item [(iii)] $\pi:D \rightarrow D$ is a linear projection onto $\theta(C)$, satisfying for each $c\in C$, $d\in D$ and $h\in H$	
		\begin{eqnarray}	
			(\pi \otimes \pi) \varDelta(h \vartriangleright d) &=& \varDelta(\pi(h \vartriangleright d)) \label{5}\\
				\pi(\pi(d)\blacktriangleleft h) &=& \varepsilon(\pi(d_1))\pi(d_2 \blacktriangleleft h) \label{6}\\
		\theta(c \leftharpoonup h)&=& \theta(c) \leftharpoonup_i h \ = \ \pi(\theta(c)\blacktriangleleft h); \label{7}
\end{eqnarray}
		\item[(iv)]$D$ is a $H$-module generated by $\theta(C)$, that is, $D = \theta(C) \blacktriangleleft H$.
	\end{enumerate}
	
\end{defn}

\begin{obs}
	 The partial action denoted by $\leftharpoonup_i$ is the induced partial action.
\end{obs}

The next subsections will be to establish a relation between the globalizations of partial module algebra and partial module coalgebra and in the sequel to determine under what assumptions a  partial $H$-module coalgebra has a globalization.

\subsection{Equivalence of Globalizations} In this subsection we will start with a globalization of a partial module coalgebra, as in Definition \ref{globalmod}. We will dualize this globalization by generating the globalization of a partial module algebra, as in Definition \ref{globallizacaoalgebra}. For this, suppose $(C, \leftharpoonup)$ a right partial $H$-module coalgebra and $(D, \blacktriangleleft)$ a right $H$-module coálgebra. Then, by Theorem \ref{teodual}, $C^*$ is a left partial $H$-module algebra via
\begin{eqnarray}\label{parcial}
(h \rightharpoondown \alpha)(c)=\alpha(c \leftharpoonup h)
\end{eqnarray}
and, analogously, $D^*$ a left $H$-module algebra via
\begin{eqnarray}\label{global}
(h \triangleright \beta)(d)=\beta(d \blacktriangleleft h).
\end{eqnarray}

Furthermore, we will consider a coalgebra monomorphism $\theta:C \rightarrow D$, the coalgebra $D= \theta(C)\blacktriangleleft H$ and $\pi:D \rightarrow D$ linear a projection onto $\theta(C)$ satisfying (\ref{5}). In this way, we can define the multiplicative monomorphism
\begin{eqnarray} \label{phi}
	\varphi: C^* &\rightarrow& D^*\\
	\alpha &\mapsto& \varphi(\alpha)= \alpha \circ \theta^{-1} \circ \pi. \nonumber
\end{eqnarray}

Under these conditions we obtain the following result.
\begin{teo}\label{dualglobal1}
 $(\theta(C)\blacktriangleleft H, \theta , \pi)$ is a globalization for $C$ if and only if
	$(H \triangleright \varphi(C^*), \varphi)$ is a globalization for $C^*$.
\end{teo}

\begin{proof}
	Suppose that $(\theta(C)\blacktriangleleft H, \theta , \pi)$ is a globalization for $C$. We will show that $(H \triangleright \varphi(C^*), \varphi)$ is a  globalization for $C^*$, as in Definition \ref{globallizacaoalgebra}. Indeed,
	
(i) $B=H \triangleright \varphi(C^*)$ is a $H$-module algebra via
		\begin{eqnarray*}
			\triangleright: H \otimes B &\rightarrow& B\\
			h \otimes (k \triangleright \varphi(\alpha)) &\mapsto& h \triangleright  (k \triangleright \varphi(\alpha)),	
		\end{eqnarray*}
		where
		$$[h \triangleright  (k \triangleright \varphi(\alpha))](d) = (k \triangleright \varphi(\alpha)) (d \blacktriangleleft h),$$
since $B=H \triangleright \varphi(C^*) \subseteq D^*$.
		
		\begin{itemize}
			\item $1_H \triangleright (k \triangleright \varphi(\alpha)) = (k \triangleright \varphi(\alpha))$, since, for all $d \in D$,
			\begin{eqnarray*}
				[1_H \triangleright  (k \triangleright \varphi(\alpha))](d) &=& (k \triangleright \varphi(\alpha)) (d \blacktriangleleft 1_H)\\
				&=& (k \triangleright \varphi(\alpha)) (d)\\
			\end{eqnarray*}	
			
			\item$h \triangleright l \triangleright (k \triangleright \varphi(\alpha)) = hl \triangleright (k \triangleright \varphi(\alpha))$, indeed, for all $d \in D$,
			\begin{eqnarray*}
				[h \triangleright l \triangleright  (k \triangleright \varphi(\alpha))](d) &=& [l \triangleright (k \triangleright \varphi(\alpha))] (d \blacktriangleleft h)\\
				&=& (k \triangleright \varphi(\alpha)) (d \blacktriangleleft h \blacktriangleleft l)\\
				&=& (k \triangleright \varphi(\alpha)) (d \blacktriangleleft h l)\\
				&=&[hl \triangleright (k \triangleright \varphi(\alpha))](d)
			\end{eqnarray*}	
			
			\item  $h \triangleright [(k \triangleright \varphi(\alpha)) (l \triangleright \varphi(\beta))]= [h_1 \triangleright (k \triangleright \varphi(\alpha))] [h_2 \triangleright (l \triangleright \varphi(\beta))]$, because, for all $d \in D$,
			\begin{eqnarray*}
				& \ &[h \triangleright [(k \triangleright \varphi(\alpha)) (l \triangleright \varphi(\beta))]](d)\\
				&=& [(k \triangleright \varphi(\alpha)) (l \triangleright \varphi(\beta))](d \blacktriangleleft h)\\
			&=&(k \triangleright \varphi(\alpha)) (d_1 \blacktriangleleft h_1)  (l \triangleright \varphi(\beta))(d_2 \blacktriangleleft h_2) \\
				&=&[h_1 \triangleright (k \triangleright \varphi(\alpha))] (d_1)   [h_2 \triangleright (l \triangleright \varphi(\beta))](d_2) \\
				&=&[[h_1 \triangleright (k \triangleright \varphi(\alpha))] [h_2 \triangleright (l \triangleright \varphi(\beta))]](d). \\
			\end{eqnarray*}
		\end{itemize}
	
(ii) $\varphi(C^*)$ is a right ideal of $B$. Indeed, for every $d \in D$,
		\begin{eqnarray*}
			[ \varphi(\beta)(h \vartriangleright \varphi(\alpha))](d)&=&\varphi(\beta)(d_1)(h \vartriangleright \varphi(\alpha))(d_2)\\
			&\stackrel{(\ref{global})}{=}&\varphi(\beta)(d_1)(\varphi(\alpha))(d_2 \blacktriangleleft h)\\
			&\stackrel{(\ref{phi})}{=}&\beta(\theta^{-1}(\pi(d_1)))\alpha(\theta^{-1}(\pi(d_2 \blacktriangleleft h)))\\
			&=&\beta(\theta^{-1}({\pi(d_1)}_2)) \varepsilon({\pi(d_1)}_2)\alpha(\theta^{-1}(\pi(d_2\blacktriangleleft h)))\\
			&=&\beta(\theta^{-1}(\pi(d_1))) \varepsilon(\pi(d_2))\alpha(\theta^{-1}(\pi(d_3\blacktriangleleft h)))\\
			&\stackrel{(\ref{6})}{=}&\beta(\theta^{-1}(\pi(d_1))) \alpha(\theta^{-1}(\pi(\pi (d_2)\blacktriangleleft h)))\\
			&\stackrel{(\ref{7})}{=}&\beta(\theta^{-1}(\pi(d_1))) \alpha(\theta^{-1}(\pi (d_2) \leftharpoonup_i h)))\\
			&\stackrel{(*)}{=}&\beta(\theta^{-1}(\pi(d_1))) \alpha(\theta^{-1}(\pi (d_2)) \leftharpoonup h)\\
			&\stackrel{(\ref{parcial})}{=}&\varphi(\beta)(d_1) (h \rightharpoondown \alpha)(\theta^{-1}(\pi (d_2)))\\
			&\stackrel{(\ref{phi})}{=}&\varphi(\beta)(d_1) \varphi(h \rightharpoondown \alpha)(d_2)\\
			&=&[\varphi(\beta) \varphi(h \rightharpoondown \alpha)](d)\\
			&=&[\varphi(\beta (h \rightharpoondown \alpha))](d),
		\end{eqnarray*}
 in (*) we used (\ref{7}) for $\theta^{-1}$. This is possible because if $\theta (c \leftharpoonup h) = \theta(c) \leftharpoonup_i h$, then applying $\theta^{-1}$ in both sides the of the equality  we obtain $(c \leftharpoonup h) = \theta^{-1}(\theta(c) \leftharpoonup_i h)$. Since for every element $d \in \theta(C)$, we know that there exists $c \in C$ such that $d = \theta(c)$ this implies that $\theta^{-1}(d)=c$, therefore, we obtain for each $d \in \theta(C)$
		\begin{eqnarray*}
			\theta^{-1}(d) \leftharpoonup h&=& (c \leftharpoonup h)\\
			&=&\theta^{-1}(\theta(c) \leftharpoonup_i h)\\
			&=&\theta^{-1}(d \leftharpoonup_i h).
		\end{eqnarray*}

(iii) $B$ is a subalgebra of $D^*$, since $\varphi(C^*)$ is a right ideal of $B$,
		
		$$(h \vartriangleright \varphi(\alpha))(k \vartriangleright \varphi(\beta)) = h_1 \vartriangleright \underbrace{(\varphi(\alpha)(S(h_2)k \vartriangleright \varphi(\beta))}_{\in \varphi(C^*)} \in B.$$
		
(iv) $\varphi(h \rightharpoondown \alpha) = h \rightharpoondown_i \varphi(\alpha) = \varphi(1_{C^*})(h \vartriangleright \varphi(\alpha))$. Indeed,
		\begin{eqnarray*}
			\varphi(1_{C^*})(h \vartriangleright \varphi(\alpha))&{=}&\varphi(\varepsilon_C)(h \vartriangleright \varphi(\alpha))\\
			&\stackrel{(*)}{=}&\varphi(\varepsilon \underbrace{(h \rightharpoondown \alpha)}_{\in C^*})\\
			&=&\varphi(h \rightharpoondown \alpha),
		\end{eqnarray*}
in $(*)$ we are using that $\varphi(\beta)(h \vartriangleright \varphi(\alpha))= \varphi(\beta(h \vartriangleright \varphi(\alpha)))$ what was proved when we showed that $\varphi(C^*) $ is an ideal of $B$.

This shows that $(H \vartriangleright \varphi(C^*), \varphi)$ is a globalization for $C^*$.
	
Conversely, note that the maps given by
\begin{eqnarray} \label{neh}
\pi(d)\leftharpoonup_{I} h = \pi(\pi(d) \blacktriangleleft h)
\end{eqnarray}
and
\begin{eqnarray}\label{eh}
h \rightharpoondown_{i} \varphi(\alpha) = \varphi(\varepsilon)(h \vartriangleright \varphi(\alpha)).
\end{eqnarray}
are linear. Since we will start from a globalization to a partial module algebra, the equation (\ref{eh}) defines an induced partial action for the global module algebra $D^*$. In contrast, in (\ref{neh}) we can not ensure that the map $\leftharpoonup_{I}$ is an induced partial action, so it will be seen only as a linear map. We will use this map to prove (\ref{7}).

Since $(H \triangleright \varphi(C^*), \varphi)$ is a globalization for $C^*$, then:
	\begin{eqnarray}\label{induzalg}
	\varphi(h \rightharpoondown \alpha )=h \rightharpoondown_i \varphi(\alpha )=\varphi(\varepsilon)(h \vartriangleright \varphi(\alpha)).
	\end{eqnarray}
	
Using (\ref{induzalg}) it is possible to show that $(h \rightharpoondown_i \varphi(\alpha))(\pi(d)) = \varphi(\alpha)(\pi(d)\leftharpoonup_{I} h). $ Indeed,
	\begin{eqnarray*}
		(h \rightharpoondown_i \varphi(\alpha ))(\pi(d))&\stackrel{(\ref{induzalg})}{=}&(\varphi(\varepsilon)(h \vartriangleright \varphi(\alpha)))(\pi(d))\\
		&=&\varphi(\varepsilon)({(\pi(d))}_1)(h \vartriangleright \varphi(\alpha))({(\pi(d))}_2)\\
		&=&\varphi(\varepsilon)(\pi(d_1))(h \vartriangleright \varphi(\alpha))(\pi(d_2))\\
		&\stackrel{(\ref{global})}{=}&\varphi(\varepsilon)(\pi(d_1))( \varphi(\alpha))(\pi(d_2)\blacktriangleleft h)\\
		&\stackrel{(\ref{phi})}{=}&\varepsilon(\theta^{-1}(\pi(\pi(d_1))))\alpha(\theta^{-1}(\pi(\pi(d_2)\blacktriangleleft h)))\\
		&{=}&\varepsilon(\theta^{-1}(\pi(d_1)))\alpha(\theta^{-1}(\pi(\pi(d_2)\blacktriangleleft h)))\\
		&=&\varepsilon(\pi(d_1))\alpha(\theta^{-1}(\pi(\pi(d_2)\blacktriangleleft h)))\\
		&=&\varepsilon({(\pi(d))}_1)\alpha(\theta^{-1}(\pi({(\pi(d))}_2 \blacktriangleleft h)))\\
		&=&\alpha(\theta^{-1}(\pi(\pi(d) \blacktriangleleft h)))\\
		&{=}&\alpha(\theta^{-1}(\pi(\pi(\pi(d) \blacktriangleleft h))))\\
		&\stackrel{(\ref{phi})}{=}&\varphi(\alpha)\pi(\pi(d) \blacktriangleleft h)\\
		&\stackrel{(\ref{neh})}{=}&\varphi(\alpha)(\pi(d) \leftharpoonup_{I} h).
	\end{eqnarray*}

	Then, we obtain the identity
	\begin{eqnarray}\label{comeco}
	(h \rightharpoondown_i \varphi(\alpha))(\pi(d)) = \varphi(\alpha)(\pi(d)\leftharpoonup_{I} h).
	\end{eqnarray}
	
	Thus, we are able to show that  $(\theta(C)\blacktriangleleft H, \theta , \pi)$ is a globalization for $C$. Note that $D$ is a right global $H$-module coalgebra via $\vartriangleleft$, given by:
	$$(\theta(c)\blacktriangleleft h) \vartriangleleft k = (\theta(c)\blacktriangleleft h) \blacktriangleleft k = \theta(c)\blacktriangleleft h  k.$$
	Besides that,
	
(i)$\pi(\pi(d)\blacktriangleleft h) = \pi(d_2 \blacktriangleleft h)\varepsilon(\pi(d_1))$, since applying $\alpha \circ \theta^{-1}$ in $\pi(\pi(d)\blacktriangleleft h)$ we obtain:
		\begin{eqnarray*}
			\alpha ( \theta^{-1}(\pi(\pi(d)\blacktriangleleft h)))&{=}& \alpha(\theta^{-1}(\pi(\pi(\pi(d)\blacktriangleleft h))))\\
			&\stackrel{(\ref{neh})}{=}&\alpha(\theta^{-1}(\pi(\pi(d) \leftharpoonup_{I} h)))\\
			&\stackrel{(\ref{phi})}{=}&\varphi(\alpha)(\pi(d) \leftharpoonup_{I} h)\\
			&\stackrel{(\ref{comeco})}{=}&(h \rightharpoondown_i \varphi(\alpha))(\pi(d))\\
			&\stackrel{(\ref{induzalg})}{=}&\varphi(h \rightharpoondown \alpha )(\pi(d))\\
			&\stackrel{(\ref{phi})}{=}&(h \rightharpoondown \alpha)(\theta^{-1}(\pi(\pi(d))))\\
			&{=}&(h \rightharpoondown \alpha)(\theta^{-1}(\pi(d)))\\
			&\stackrel{(\ref{phi})}{=}&\varphi(h \rightharpoondown \alpha )(d)\\
			&\stackrel{(\ref{induzalg})}{=}&(\varphi(\varepsilon)(h \vartriangleright \varphi(\alpha)))(d)\\
			&{=}&(\varphi(\varepsilon)(d_1)(h \vartriangleright \varphi(\alpha)))(d_2)\\
			&\stackrel{(\ref{phi})}{=}&\varepsilon(\theta^{-1}(\pi(d_1)))(h \vartriangleright \varphi(\alpha)))(d_2)\\
			&\stackrel{(\ref{global})}{=}&\varepsilon(\theta^{-1}(\pi(d_1)))(\varphi(\alpha))(d_2 \blacktriangleleft h)\\
			&=&\varepsilon(\pi(d_1))\alpha(\theta^{-1}(\pi(d_2 \blacktriangleleft h))).\\
		\end{eqnarray*}

	Now we note that since the equality above holds for all $\alpha \in C^*$, we obtain
		$$\theta^{-1}(\pi(\pi(d)\blacktriangleleft h)) = \varepsilon(\pi(d_1))\theta^{-1}(\pi(d_2 \blacktriangleleft h))$$
		and applying $\theta$ in both sides of this equality we have		$$\pi(\pi(d)\blacktriangleleft h) = \varepsilon(\pi(d_1))\pi(d_2 \blacktriangleleft h),$$
because, since $Im \pi = \theta(C) = Im \theta = Dom \theta^{-1} $, then, $\theta \circ \theta^{-1}= Id_{|_{Dom \theta^{-1}}}=Id_{|_{Im \pi}}.$

(ii) $\theta(c \leftharpoonup h)= \theta(c) \leftharpoonup_{i} h$, because applying $\alpha \circ \theta^{-1}$ in the right side we obtain,
		\begin{eqnarray*}
			(\alpha \circ \theta^{-1})(\theta(c) \leftharpoonup_{i} h)&\stackrel{(\ref{neh})}{=}&(\alpha \circ \theta^{-1})(\pi (\theta(c) \blacktriangleleft  h))\\
			&\stackrel{(\ref{phi})}{=}& \varphi (\alpha)(\theta(c) \blacktriangleleft  h)\\
			&=&(h \vartriangleright \varphi (\alpha))(\theta(c))\\
			&=&(h \vartriangleright \varphi (\alpha))(\theta(c_2))\varepsilon(\theta^{-1}(\theta(c_1)))\\
			&{=}&(h \vartriangleright \varphi (\alpha))(\theta(c_2))\varepsilon(\theta^{-1}(\pi (\theta(c_1))))\\
			&\stackrel{(\ref{phi})}{=}&(h \vartriangleright \varphi (\alpha))(\theta(c_2)) \varphi (\varepsilon) (\theta(c_1))\\
			&{=}&\varphi (\varepsilon) ({\theta(c)}_1) (h \vartriangleright \varphi (\alpha))({\theta(c)}_2) \\
			&{=}&(\varphi (\varepsilon)  (h \vartriangleright \varphi (\alpha)))({\theta(c)}) \\
			&\stackrel{(\ref{eh})}{=}& (h \rightharpoondown_{i} \varphi(\alpha)) ({\theta(c)}) \\
			&{=}& \varphi(h \rightharpoondown \alpha) ({\theta(c)}) \\
			&\stackrel{(\ref{phi})}{=}& (h \rightharpoondown \alpha) ( \theta^{-1} (\pi ({\theta(c)}))) \\
			&{=}& (h \rightharpoondown \alpha) ( \theta^{-1} ({\theta(c)})) \\
			&=& (h \rightharpoondown \alpha) (c) \\
			&\stackrel{(\ref{parcial})}{=}& \alpha(c \leftharpoonup h) \\
			&=& \alpha(\theta^{-1}(\theta (c \leftharpoonup h))),
		\end{eqnarray*}
Besides that, it is necessary call the attention that the first equality holds by the fact that we already shown that the linear map $\leftharpoonup_{i}$ is the induced action.
		
In this way, $(\alpha \circ \theta^{-1})(\theta(c) \leftharpoonup_{i} h) =  \alpha(\theta^{-1}(\theta (c \leftharpoonup h)))$, for all $\alpha \in C^{*}$, what implies that $(\theta(c) \leftharpoonup_{i} h) =  (\theta (c \leftharpoonup h))$.

\end{proof}

\subsection{Constructing a Globalization}
\quad
Consider $C$ a right partial $H$-module coalgebra. In \cite{Glauberglobalizations}, it was constructed the standart globalization for $C$, with $H$ a Hopf algebra, via the maps
\begin{eqnarray*}
	\theta: C &\rightarrow& C \otimes H \ \ \ \mbox{ and }  \ \ \ 	\pi: C \otimes H \rightarrow C \otimes H\\
	c &\mapsto& c \otimes 1_H \ \ \ \ \ \ \ \ \ \ \ \ \ \ \ \ \ \ \ 	c \otimes h \mapsto c \leftharpoonup h \otimes 1_H
\end{eqnarray*}
where $\theta$ is a coalgebra monomorphism and $\pi$ is a comultiplicative projection onto $\theta(C)$. However, for the case of weak Hopf algebra, since the element $1_H$ does not satisfies $\Delta(1_H) = 1_H \otimes 1_H$ the maps $\theta$ and $\pi$ do not satisfied such properties. Besides that, $\theta(c)=c \otimes h$ is a coalgebra monomorphsm if and only if $\Delta(h)=h \otimes h$. Therefore, it is natural consider an additional hypothesis for this weak Hopf algebra, as we can see in the following result.

\begin{teo}[Globalization Theorem]\label{globalizacao}
	Suppose that $C$ is a right partial $H$-module coalgebra such that there exists an element $e \in H$ that satisfies
	\begin{eqnarray}
	\Delta(e) = e \otimes e
	\end{eqnarray}
	\begin{eqnarray}\label{issoeum}
	c \leftharpoonup he= c \leftharpoonup h,
	\end{eqnarray}
	for all $c \in C$ and $h \in H$. Then, $C$ has a globalization.
\end{teo}

\begin{proof}
First of all, note that (\ref{issoeum}) implies that $c = c \leftharpoonup e$, for all $c \in C$. Moreover, for all $c \in C$ and $h \in H$,
\begin{eqnarray*}
c \leftharpoonup h &=& c \leftharpoonup e \leftharpoonup h\\
&{=}& \varepsilon(c_1 \leftharpoonup e) c_2 \leftharpoonup eh\\
&{=}& \varepsilon(c_1) c_2 \leftharpoonup eh\\
&{=}& c \leftharpoonup eh.\\
\end{eqnarray*}
Define
	\begin{eqnarray*}
		\theta: C &\rightarrow& C \otimes eH\\
		c &\mapsto& c \otimes e.
	\end{eqnarray*}
Note that $\theta$ is injective because $\Delta(e) = e \otimes e$, then, $\varepsilon(e)= 1_{\Bbbk}$. Moreover, $\theta$ is a coalgebra homomorphism, because $\varepsilon(\theta(c))= \varepsilon(c \otimes e) = \varepsilon(c) \varepsilon(e)= \varepsilon(c)$, and
	\begin{eqnarray*}
		(\theta \otimes \theta)(\Delta(c))&=&\theta(c_1) \otimes \theta (c_2)\\
		&=&c_1 \otimes e \otimes c_2 \otimes e\\
		&=& \Delta(\theta (c)).
	\end{eqnarray*}
	
We define
	\begin{eqnarray*}
		\pi: C \otimes eH &\rightarrow& C \otimes eH\\
		c \otimes eh &\mapsto& c \leftharpoonup eh \otimes e,
	\end{eqnarray*}
then $\pi$ is a comultiplicative projection onto $\theta(C)$. Indeed,
	
	\begin{itemize}
		\item $(\pi \otimes \pi)(\Delta(c \otimes eh)) = \Delta (\pi (c \otimes eh))$, since
		\begin{eqnarray*}
			(\pi \otimes \pi)(\Delta(c \otimes eh))&=&(\pi \otimes \pi)((c \otimes eh)_1 \otimes (c \otimes eh)_2)\\
			&=&\pi(c_1 \leftharpoonup eh_1) \otimes \pi (c_2 \leftharpoonup eh_2)\\
			&=& c_1 \leftharpoonup eh_1 \otimes e \otimes c_2 \leftharpoonup eh_2 \otimes e\\
			&=&\Delta(c \leftharpoonup eh \otimes e)\\
			&=&\Delta (\pi (c \otimes eh)).
		\end{eqnarray*}	
		
		\item $\pi(\pi(c \otimes eh)) = \pi(c \otimes eh)$, because
		\begin{eqnarray*}
			\pi(\pi(c \otimes eh))&=& (c \leftharpoonup eh ) \leftharpoonup e \otimes e\\
			&=& \varepsilon(c_1 \leftharpoonup eh_1)(c_2 \leftharpoonup eh_2 e) \otimes e\\
			&\stackrel{(\ref{issoeum})}{=}&\varepsilon(c_1 \leftharpoonup eh_1)(c_2 \leftharpoonup eh_2 ) \otimes e\\
			&=&c \leftharpoonup eh \otimes e\\
			&=&\pi(c \otimes eh).
		\end{eqnarray*}
		
		\item $\pi(\theta(c))= \theta(c)$, since $\theta (c) = c \otimes e = c \leftharpoonup 1_H \otimes e\stackrel{(\ref{issoeum})}{=} c \leftharpoonup e1_H \otimes e  = \pi(c \otimes e1_H)$, then:
		\begin{eqnarray*}
			\pi(\theta(c)) = \pi(\pi(c \otimes e1_H)) = \pi(c \otimes e1_H)= \theta(c).
		\end{eqnarray*}
		
		\item $\pi(\pi(c \otimes eh) \blacktriangleleft k) = \varepsilon(\pi(c_1 \otimes eh_1))\pi((c_2 \otimes eh_2) \blacktriangleleft k)$, where $C \otimes eH$ is seen as a right $H$-module coalgebra via the product on the last factor. It is enough to note that
		\begin{eqnarray*}
			\pi(\pi(c \otimes eh) \blacktriangleleft k)
			&=& \pi((c \leftharpoonup eh \otimes e) \blacktriangleleft k)\\
			&=& \pi(c \leftharpoonup eh \otimes ek) \\
			&=&c \leftharpoonup eh \leftharpoonup ek  \otimes e\\
			&\stackrel{(\ref{issoeum})}{=}&c \leftharpoonup eh \leftharpoonup k  \otimes e\\
			&=& \varepsilon(c_1 \leftharpoonup eh_1)(c_2 \leftharpoonup eh_2 k)  \otimes e\\
			&{=}&\varepsilon(c_1 \leftharpoonup eh_1 \otimes e)(c_2 \leftharpoonup eh_2 k)  \otimes e\\
			&=&\varepsilon(\pi(c_1 \otimes eh_1)) \pi (c_2 \otimes eh_2 k)\\
			&=& \varepsilon(\pi(c_1 \otimes eh_1))\pi((c_2 \otimes eh_2)\blacktriangleleft k).
		\end{eqnarray*}
		
		\item $\theta (c) \leftharpoonup_i h = \theta (c \leftharpoonup h) $, because
		\begin{eqnarray*}
			\theta (c) \leftharpoonup_i h &=& \pi(\theta (c) \blacktriangleleft h)\\
			&=&\pi((c \otimes e) \blacktriangleleft h)\\
			&=&\pi(c \otimes eh)\\
			&=&c \leftharpoonup eh \otimes e\\
			&\stackrel{(\ref{issoeum})}{=}&c \leftharpoonup h \otimes e\\
			&=& \theta (c \leftharpoonup h).
		\end{eqnarray*}
		
	\end{itemize}
Therefore, $\theta(C)\blacktriangleleft H = C \otimes eH$ is a globalization for $C$ as a partial $H$-module coalgebra.
\end{proof}

\begin{ex}

(i) Consider $\Bbbk \mathcal{G}$ with $\mathcal{G}$ a groupoid given by the  disjoint union of the groups $G_1$ e $G_2$. Consider the linear map
		$$ \lambda(\delta_g) =\left \{ \begin{array}{rl}  1_{\Bbbk}, \ if \ g = e_1 \\
		0, if \ g \neq e_1,
		\end{array}
		\right. $$
		where $e_1$ denotes the identity of $G_1$. In this way, it is enough to take $\delta_{e_1}$ and the partial action on a coalgebra $C$ given by $c \leftharpoonup \delta_g = c \lambda (\delta_g)$, for every $c \in C$ . Since $\lambda(\delta_g) = \lambda(\delta_{e_1} \delta_g) = \lambda(\delta_g \delta_{e_1})$, it follows that this partial action is globalizable.
(ii) Consider $\Bbbk \mathcal{G}$ a groupoid algebra and fix $e$ an element of $\mathcal{G}_0$. Then,  
	$$ \lambda(\delta_g) =\left \{ \begin{array}{rl}  1_{\Bbbk}, \ if \ g \in \mathcal{G}_e\\
0, if \ g \notin \mathcal{G}_e,
\end{array}
\right. $$
defines a partial module coalgebra of $\Bbbk \mathcal{G}$ on a coalgebra $C$, as in Example \ref{exlambda}(ii). Since $\lambda(\delta_g) = \lambda(\delta_{e} \delta_g) = \lambda(\delta_g \delta_{e})$, it follows that this partial action is globalizable.
\end{ex}

\section{Acknowledgments}
The authors would like to thank to Glauber Quadros and Felipe Castro for their solicitude all the times they have been willing to discuss about partial action theory. Also to Eliezer Batista for the suggestions to improve this work. And finally to Antonio Paques who kindly reviewed this article giving a fundamental contribution to its finalization.

\addcontentsline{toc}{chapter}{Referências Bibliográficas}


\begin{thebibliography}{99}
	
	\bibitem{Ab} M. Alves, E. Batista, Enveloping actions for partial Hopf actions. Communications in Algebra 38, 2872 - 2902, 2010.
	
	
	\bibitem{Bagio} D. Bagio, A. Paques, Partial groupoide actions: globalization, Morita theory, and Galois theory, Communications in Algebra 40, (10) 3658 - 3678, 2012.
	
	\bibitem{Batista} E. Batista, J. Vercruysse, Dual constructions for partial actions of Hopf algebras, Journal of Pure and Applied Algebra 220 (2) , 518-559, 2016.
	
	\bibitem{Bohminicio} G. Böhm, Doi-Hopf modules over weak Hopf algebras. Communications in Algebra, 28, 4687 - 4698, 2000.
	
	
	\bibitem{Bohmexemplo} G. Böhm, J. Gómes-Torrecillas, On the Double Crossed Product of Weak Hopf Algebras. Contemporary Mathematics 585, 153-174, 1999.
	
\bibitem{Bohm} G. Böhm, F. Nill, K. Szlachányi, Weak Hopf Algebras I: Integral Theory and $C^*$-Structure. Journal of Algebra 221 (2), 385-438, 1999.
	
	
	\bibitem{Caenepeel} S. Caenepeel, E. De Groot, Modules Over Weak Entwining Structures in "New trends Hopf Algebra Theory", Contemporary Mathematics 267 , 31-54, 2000.
	
	\bibitem{CaenJanssen}S. Caenepeel, K. Janssen, Partial (Co)Actions of hopf algebras and Partial Hopf-Galois Theory, Communications in Algebra 36:8, 2923-2946, 2008.
	
	\bibitem{Felipeweak} F. Castro, A. Paques, G. Quadros, A. Sant'Ana, Partial actions of weak Hopf algebras: smash products, globalization and Morita theory, Journal of Pure and Applied Algebra 29, 5511 - 5538, 2015.
	
	\bibitem{Glauberglobalizations} F. Castro, G. Quadros, Globalizations for partial (co)actions on coalgebras, http://arxiv.org/abs/1510.01388v1.
	
	\bibitem{Dokuchaev} M. Dokuchaev, R. Exel, Associativity of Crossed Products by Partial Actions, Enveloping Actions and Partial Representations,  Trans. Amer. Math. Soc. 357 (5)  1931-1952, 2005.
	
	\bibitem{michamiguelpaques} M. Dokuchaev, M. Ferrero, A. Paques, Partial Actions and Galois Theory, Journal of Pure
	and Applied Algebra 208 (1) 77-87, 2007.	
	
	\bibitem{Exelp}R. Exel, Circle Actions on $C^*-$Algebras, Partial Automorphisms and Generalized PimsnerVoiculescu Exect Sequences, J. Funct. Anal.122 (3) 361-401, 1994.
	
	
	
	
	\bibitem{Sweedler} M. Sweedler, Hopf algebras, Mathematics Lecture Note Series, W. A. Benjamin, Inc., New York, 1969.
	
	\bibitem{Wang} Yu. Wang, L. Yu. Zhang, The Structure Theorem for Weak Module Coalgebras, Mathematical Notes 88 (1), 3 - 17, 2010.
	
\end{thebibliography}
\end{document}